\documentclass[reqno]{amsart}

\usepackage{graphicx,pdfsync,amssymb, latexsym,amsfonts,amsbsy,color,subcaption,esint,mathtools,enumerate,nicefrac}

\usepackage{graphicx,pdfsync,amssymb, latexsym,amsfonts,amsbsy,color,subcaption,esint,mathtools,nicefrac,tikz}

\usepackage{parskip}

\usepackage{calligra}
\usepackage{mathtools}
\newtagform{total}{\text{\calligra  \Large Total dissipation anomaly }(}{)}
\newtagform{partial}{\text{\calligra  \Large Partial dissipation anomaly }(}{)}
\newtagform{partial_or_no}{\text{\calligra  \Large Partial or no dissipation anomaly }(}{)}
\newtagform{partial_total_or_no}{\text{\calligra  \Large Partial, total,  or no dissipation anomaly }(}{)}
\newtagform{noDA}{\text{\calligra  \Large No dissipation anomaly }(}{)}
\newtagform{LofE}{\text{\calligra  \Large Energy limit }(}{)}
\newtagform{EofL}{\text{\calligra  \Large Energy of the limit }(}{)}
\newtagform{LofD}{\text{\calligra  \Large Dissipation limit }(}{)}

\usepackage{enumitem}
\setitemize{itemsep=0.5em}
\setenumerate{itemsep=0.5em}

\usepackage{hyperref}
\hypersetup{
colorlinks=true,
linkcolor=blue,
filecolor=blue,      
urlcolor=blue,
citecolor=blue,
}

\usepackage[utf8]{inputenc}
\usepackage[T1]{fontenc}

\usepackage{tikz,caption}
\usetikzlibrary{patterns}

\usepackage[top=1.1in, bottom=1in, left=1in, right=1in]{geometry}

\DeclareMathOperator{\Supp }{supp}

\newtheorem{theorem}{Theorem}[section]
\newtheorem{lemma}[theorem]{Lemma}

\newtheorem{definition}[theorem]{Definition}
\newtheorem{assumption}[theorem]{Assumption}

\newtheorem{remark}[theorem]{Remark}

\usepackage{times}
\usepackage{setspace}
\def \TT  {\mathbb{T}} 
\def \RR {\mathbb{R}}  
\def \NN {\mathbb{N}}  
\def \ZZ {\mathbb{Z}}  

\def \l {\lambda}
\def \L {\Lambda}

\def \p {\partial}

\newcommand{\comment}[1]{}

\numberwithin{equation}{section}

\begin{document}

\title[Dissipation anomaly and anomalous dissipation]{Dissipation anomaly and anomalous dissipation in incompressible fluid flows}

\author{Alexey Cheskidov}
\address[Alexey Cheskidov]{Institute for Theoretical Sciences, Westlake University}

\email{cheskidov@westlake.edu.cn}


\date{\today}

\begin{abstract}
Dissipation anomaly, a phenomenon predicted by Kolmogorov's theory of turbulence, is the persistence of a non-vanishing energy dissipation
for solutions of the Navier-Stokes equations as the viscosity goes to zero. Anomalous dissipation, predicted by Onsager, is a failure of solutions of the limiting Euler equations to preserve the energy balance. Motivated by a recent dissipation anomaly construction for the 3D Navier-Stokes equations by Bru\`e and De Lellis (2023), we prove the existence of various scenarios in the limit of vanishing viscosity: the total and partial loss of the energy due to dissipation anomaly, absolutely continuous dissipation anomaly, anomalous dissipation without dissipation anomaly, and the existence of infinitely many limiting solutions of the Euler equations in the limit of vanishing viscosity.  We also discover a relation between dissipation anomaly and the discontinuity of the energy of the limit solution.
Finally, expanding on the obtained total dissipation anomaly construction, we show the existence of dissipation anomaly for long time averages, relevant for turbulent flows, proving that the Doering-Foias (2002) upper bound is sharp.
\end{abstract}


\maketitle

\section{Introduction}
Dissipation anomaly, predicted by Kolmogorov’s theory of turbulence \cite{1371131421197965440}, postulates that
\begin{equation} \label{eq:DA-intro}
\frac{\epsilon \ell}{U^3} = \mathcal{O}(Re^0) \quad \text{as} \quad Re \to \infty,
\end{equation}
where $\epsilon$ is the total energy dissipation rate per unit mass, $\ell$  is an integral length scale in the flow, $U$ is a turbulent characteristic velocity, and $Re = U\ell/\nu$ is the Reynolds number with $\nu$ denoting the kinematic viscosity coefficient.
This law has been extensively verified experimentally for turbulent fluid flows as well as numerically for solutions of the 3D Navier-Stokes equations at high Reynolds numbers  \cite{10.1063/1.864731, 10.1063/1.869575, 10.1063/1.1539855, 10.1063/1.1445422, doi:10.1146/annurev-fluid-010814-014637}. On the other hand, anomalous dissipation, predicted by Onsager \cite{Onsager}, is a failure of solutions of the limiting Euler equations to satisfy the energy balance, necessary for the validity of the inviscid equations in turbulent regimes.
  The precise regularity threshold conjectured by Onsager for inviscid solutions to sustain the energy cascade has been a focus of intensive mathematical investigations for decades, see Subsection~\ref{Subsec:Previous_results}. 



In turbulent flows, the energy injected at forced low modes (large scales) cascades to small scales through the inertial range where viscous effects are negligible, and only dissipates above Kolmogorov's dissipation wavenumber that goes to infinity as $Re \to \infty$. The persistence of the energy flux through the inertial range is what constitutes dissipation anomaly for viscous fluid flows as well as anomalous dissipation for the limiting inviscid flows. Inspired by the construction due to Bru\`e and De Lellis \cite{MR4595604}, we first analyze these intrinsically linked phenomena on a finite time interval and prove the existence of various scenarios in the limit of vanishing viscosity, ranging from the total dissipation anomaly to a pathological one where anomalous dissipation occurs without dissipation anomaly (see Theorem~\ref{thm:main}), as well as the existence of infinitely many limiting solutions of the Euler equations in the limit of vanishing viscosity (see Theorem~\ref{thm:main-2}).

Then we turn to long time averages of the energy dissipation, appropriate for describing turbulent fluid flows, for which it was shown by Foias and Doering \cite{MR1928940} that the upper bound
\begin{equation} \label{b:bound_on_drag_coefficient-intro}
\frac{\epsilon \ell}{U^3} \leq c_1 +c_2 Re^{-1},
\end{equation}
holds with dimensionless coefficients $c_1$ and $c_2$ dependent only on the shape of the force. Extending the obtained total dissipation anomaly construction to the whole time interval, we show the existence of dissipation anomaly for long time averages, proving that the Doering-Foias upper bound \eqref{b:bound_on_drag_coefficient-intro} is sharp, see Theorem~\ref{thm:DA_in_FT}.

\subsection{Dissipation anomaly for finite time averages}

We consider the vanishing viscosity limit of solutions to the 3D Navier-Stokes equations
\begin{equation}\label{eq:NSE} \tag{NSE}
\begin{cases}
\p_t u^{\nu}  + (u^{\nu} \cdot \nabla)u^{\nu}  = -\nabla p^{\nu} + \nu  \Delta u^{\nu}+ f^{\nu} &\\
\nabla \cdot u^{\nu} = 0&\\
u^{\nu}(0) = u_{\mathrm{in}}
\end{cases}
\end{equation}
posed on $\TT^3$. Here $\nu > 0$ is the viscosity coefficient, the initial data $u_{\mathrm{in}}\in L^2$ is divergence free, and we consider weak solutions on $[0,2]$ satisfying the energy equality
\begin{equation} \label{eq:EE}
\|u^{\nu}(t)\|_{L^2}^2 =  \|u_{\mathrm{in}}\|_{L^2}^2 - 2\nu\int_{0}^{t} \|\nabla u^{\nu}(\tau)\|_{L^2}^2 \, d\tau + 2\int_{0}^{t} (f^{\nu},u^{\nu}) \, d\tau,
\end{equation}
for all $t \in[0,2]$. In fact, in all the constructions the initial data $u_{\mathrm{in}}$, the force $f^\nu$, and solutions of \eqref{eq:NSE} will be smooth.
To ensure the convergence of the work done by the force $f^\nu$, we require
$f^{\nu} \to f$ in $L^1(0,2;L^2)$ as $\nu \to 0$. The goal is to analyze potential dissipation anomaly in the limit of vanishing viscosity when $u^{\nu} \to u$  in $C([0,2];L^2_w),$ and $u$ is a weak solution of the Euler equations.

We will be examining the following objects:
\usetagform{LofE}
\begin{flalign} 
\label{intro:DefE} \qquad E(t):={\displaystyle \liminf_{\nu\to 0}}\|u^{\nu}(t)\|_{L^2}^2&&
\end{flalign}
\vspace{-1em}
\usetagform{LofD}
\begin{flalign}
\qquad D(t):={\displaystyle \limsup_{\nu \to 0} 2\nu\int_{0}^t \|\nabla u^{\nu}(\tau)\|_{L^2}^2 \, d\tau}&&
\end{flalign}
\vspace{-1em}
\usetagform{EofL}
\begin{flalign}
\qquad \|u(t)\|_{L^2}^2&&
\end{flalign}
\usetagform{default}

The assumption $f^{\nu} \to f$ in $L^1(0,2;L^2)$ together with the weak convergence of $u^\nu$ to $u$ implies the convergence of the work done by the force, which we denote by $W$:
\[
W(t) := \lim_{\nu \to 0} 2\int_0^t (f^{\nu},u^{\nu}) \, d\tau = 2\int_0^t (f,u) \, d\tau.
\]
In other words, there is no {\em anomalous work}, or work on vanishing scales, see \cite{MR4283701}. Hence taking the limit of the energy equality \eqref{eq:EE} and using the weak convergence of $u^\nu$ to $u$ we obtain
\begin{equation} \label{eq:Intro_Enegry_Balance}
0 \leq \|u(t)\|_{L^2}^2 \leq E(t)= \|u_{\mathrm{in}}\|_{L^2}^2 -D(t)+W(t),
\end{equation}
and, in particular,
\begin{equation} \label{eq:Intro_D(t)_bounds}
0\leq D(t) \leq \|u_{\mathrm{in}}\|_{L^2}^2+W(t).
\end{equation}

We will study two related phenomena. 

\begin{definition} \label{Def:Dissipation_Anomaly_Intro}
With $D(t)$ and $W(t)$ defined as above, we say \\
\vspace{-1em}
\begin{itemize}
\item The family of solutions to \eqref{eq:NSE} $u^\nu$ exhibits {\bf \em dissipation anomaly} on $[0,t]$ if $D(t)>0$.
\item
The limiting solution $u(t)$ exhibits {\bf \em anomalous dissipation} on $[0,t]$ if $ \|u_{\mathrm{in}}\|_{L^2}^2 +W(t) - \|u(t)\|_{L^2}^2>0$.
\end{itemize}
\end{definition}
First, notice that \eqref{eq:Intro_Enegry_Balance} yields
\[
\|u_{\mathrm{in}}\|_{L^2}^2 +W(t) - \|u(t)\|_{L^2}^2 \geq D(t),
\]
and hence {\bf \em dissipation anomaly} automatically implies {\bf \em anomalous dissipation}. We will analyze various possibilities for the above inequality and, among other things, show  that the converse is not true. We will follow the approach of Bru\`e and De Lellis \cite{MR4595604} who recently constructed the first example of {\bf \em dissipation anomaly} in the sense of Definition~\ref{Def:Dissipation_Anomaly_Intro}. 


We will examine the following 

 \noindent
{\bf Questions:}
{
\begin{enumerate}
\item Can $D(t)$  achieve all the possible values allowed by  inequalities in \eqref{eq:Intro_D(t)_bounds}? \label{Q1}
\item Can $\|u(t)\|_{L^2}^2$ achieve all the possible values allowed by  inequalities in \eqref{eq:Intro_Enegry_Balance}?  \label{Q2}
\item Can $E(t)$ and $D(t)$ be continuous (or even absolutely continuous) with nontrivial $D(t)$? \label{Q3}
\item Can $\|u(t)\|_{L^2}$ be continuous  with nontrivial $D(t)$? \label{Q:Cont_u}
\item Does {\bf \em Anomalous Dissipation} imply {\bf \em Dissipation Anomaly}? \label{Q5}
\item Can there be infinitely many limiting solutions of the Euler equations in the limit of vanishing viscosity (violating the vanishing viscosity selection principle)? \label{Q6}
\end{enumerate}
}

\vspace{0.1in}

We summarize the results addressing Questions~\eqref{Q1}, \eqref{Q2}, \eqref{Q3}, \eqref{Q5}, and \eqref{Q6} below. See Theorems~\ref{thm:main} and \ref{thm:main-2} for precise statements. A conditional (negative) answer to Question \eqref{Q:Cont_u} is given in Theorems~\ref{thm:discont-intro} and \ref{thm:Discont_of_u}.

\vspace{0.1in}

\noindent
{\bf Forward energy cascade before $t=1$: arbitrary {\em dissipation anomaly}, discontinuous $E(t)$ -- Figure~\ref{fig:Thm1-convergence}.}

{\em
There is a family  of solutions to \eqref{eq:NSE} $\{u^\nu(t)\}_{\nu}$ on $[0,2]$ with the same initial data $u_{\mathrm{in}}$
with $\|u_{\mathrm{in}}\|_{L^2}^2=1$ 
that converges weakly to a solution of the Euler equation $u(t)$ that satisfies the energy equality up to $t=1$ (where the work of the force $W(1)=0$), but looses all the energy at $t=1$,
\[
u(t)=0, \quad \forall t\geq 1,
\] 
such that for every $e\in[0,1]$ there is a subsequence with the dissipation anomaly and energy limit being
\[
\qquad D(1)=e, \qquad D(t)=1, \qquad E(1)=1-e, \qquad E(t)=0,  \quad \forall t>1.
\] 
}
\vspace{0.1in}

\noindent
{\bf Forward energy cascade before $t=1$: arbitrary {\em dissipation anomaly}, absolutely continuous $E(t)$ -- Figure~\ref{fig:Thm1-energy_profiles}.}

{\em
For  
every $e\in[0,1)$ there is a subsequence of the above family $\{u^\nu\}$ with
\[
D(2)=e, \qquad E(2)=1-e, \qquad E \in C([0,2]).
\] 
In particular, when $e=0$, there is no {\bf \em dissipation anomaly} in the presence of {\bf \em anomalous dissipation}.
}
\vspace{0.1in}

\noindent
{\bf Forward energy cascade before $t=1$ followed by inverse energy cascade after $t=1$: infinitely many limiting solutions -- Figure~\ref{fig:Thm2-convergence}.}

{\em
There is initial data $u_{\mathrm{in}}$ that gives rise to a family of solutions to \eqref{eq:NSE} $\{u^\nu(t)\}_{\nu}$ and a sequence of solutions of the Euler equations $u_n(t)$ on $[0,2]$ with increasing $\|u_n(2)\|_{L^2}$ and such that
\[
u^{\nu^n_j} \to u_n, \qquad \text{in} \qquad C_w([0,2];L^2),
\]
as $j \to \infty$, for some subsequence  $\nu^n_j \to 0$.
}

\subsection{Dissipation anomaly for long time averages} \label{Subsec:long-time_averages}
In mathematical theory of turbulence it is appropriate to take long time averages as they reveal the underlying statistical properties of turbulent flows.
Since the diameter of the weak global attractor in $L^2$ in general grows  as the viscosity goes to zero (or the amplitude of the force goes to infinity), the dissipation anomaly law has to be stated in terms of dimensionless variables in this case. Also, we do not expect a body force, or even its shape, to stay fixed in an experiment. So we start with any sequence of dimensionless shapes of the force $\phi^j$  converging in $L^2$ to some nontrivial divergence-free $\phi$.
Now consider Leray-Hopf weak solutions $u^\nu$ of \eqref{eq:NSE} on torus $[0,\ell]^3$ with force $f(x)=F \phi^j(\ell^{-1} x)$, $F \in \RR$. For any such solution, we can use a generalized Banach limit to define the energy dissipation $\epsilon$, characteristic velocity $U$, and the Reynolds number $Re$ as follows: 
\begin{equation} \label{def:eps,U,Re}
\begin{gathered}
\epsilon=\langle \nu |\nabla u^\nu|^2 \rangle  :=\underset{T\to \infty}{\mathrm{Lim}} \frac{1}{T\ell^3}\int_0^T \nu \|\nabla u^\nu\|_{L^2}^2\, dt,\qquad
U=\langle |u^\nu|^2 \rangle^{\frac12}:=\left(\underset{T\to \infty}{\mathrm{Lim}} \frac{1}{T\ell^3}\int_0^T \|u^\nu\|_{L^2}^2\, dt\right)^{1/2},\\
Re := \frac{U\ell}{\nu}.
\end{gathered}
\end{equation}
Then a generalization of the argument by Foias \cite{MR1467006} and Doering and Foias \cite{MR1928940} gives an upper bound on the normalized energy dissipation (also called the drag coefficient, see Frisch \cite{Frisch}),
\begin{equation} \label{b:bound_on_drag_coefficient}
\frac{\epsilon \ell}{U^3} \leq c_1 +c_2 Re^{-1},
\end{equation}
where the dimensionless coefficients $c_1$ and $c_2$ depend only on the limit shape function $\phi$. As predicted by Kolmogorov's theory of turbulence and observed in experiments and direct numerical simulations, bound \eqref{b:bound_on_drag_coefficient} is expected to be saturated by turbulent flows. It is easy to construct examples of laminar flows, for instance, steady states,  where $\frac{\epsilon \ell}{U^3} \sim Re^{-1}$ for all $Re$. However, there is no known example of a sequence of the NSE solutions with $Re \to \infty$, where bound \eqref{b:bound_on_drag_coefficient} is attained at high Reynolds numbers. i.e.,
\[
\limsup_{Re \to \infty} \frac{\epsilon \ell}{U^3} >0.
\]

On the other hand, it is hard to expect the body force to stay constant in time throughout an experiment. So in a more realistic framework, we consider a sequence of dimensionless shapes of the force $\phi^j$  converging in $C(\RR; L^2(\TT^3))$ to some  $\phi$. Then following the argument of Doering and Foias \cite{MR1928940} we can show
\begin{theorem} \label{thm:DF}
For any nontrivial divergence-free $\phi \in L^\infty(\RR; L^2(\TT^3))$ there exist constants $c_1$ and $c_2$, such that for every Leray-Hopf weak solution $u^\nu$ of \eqref{eq:NSE} with arbitrary divergence-free finite energy initial data and force $f(x,t)=F \phi^j(\ell^{-1} x,t)$ with arbitrary amplitude $F \in \RR$ and shapes $\phi^j \to \phi$ in $C(\RR; L^2(\TT^3))$,
the energy dissipation satisfies
\begin{equation} \label{bound_on_drag_coefficient_thm}
\frac{\epsilon \ell}{U^3} \leq c_1 +c_2 Re^{-1},
\end{equation}
where $\epsilon$, $U$, and $Re$ are defined as in \eqref{def:eps,U,Re}.
\end{theorem}
Now, using the total dissipation anomaly construction in Theorem~\ref{thm:main}, we can prove that bound \eqref{bound_on_drag_coefficient_thm} is attained by some sequence of solutions to \eqref{eq:NSE}.
\begin{theorem} \label{thm:DA_in_FT}
For every $\ell>0$, $U>0$, and $c>0$ there is a sequence of smooth time-periodic solutions to \eqref{eq:NSE} $u^{\nu_j} \in L^\infty(\RR,L^2([0,\ell]^3))$ with viscosity $\nu_j \to 0$, force $f^{\nu_j} \to f$ in $C(\RR;L^2)$ as $j \to 0$, such that \eqref{def:eps,U,Re} hold, as well as
\begin{equation} \label{eq:long_time_DA}
\lim_{j\to \infty} \frac{\epsilon \ell}{U^3} = c.
\end{equation}
\end{theorem}

Solutions $u^{\nu_j}$ in this theorem are ancient (complete), as they are in $L^\infty(\RR;L^2)$, so they stay on the weak pullback attractors $\mathcal{A}_{\mathrm w}(t)$ of \eqref{eq:NSE} for all values of viscosity as it goes to zero (and $Re \to \infty$). The weak pullback attractor captures all of the asymptotic dynamics of Leray-Hopf solutions, see \cite{MR3331677}. For low Reynolds number $Re$ (large viscosity), the weak pullback attractor consists of only one trajectory, $\mathcal{A}_{\mathrm w}(t)=\{u^{\nu_j}(t)\}$, see \cite{MR3383921}. So the constructed solutions $u^{\nu_j}(t)$ attract all the other trajectories in the pullback sense in this case. However, when $Re$ becomes large, $\mathcal{A}_{\mathrm w}(t)$ is expected to grow, and $u^{\nu_j}(t)$ might lose stability. So it is not clear whether \eqref{eq:long_time_DA} still holds at high Reynolds numbers if we also take some ensemble average. Hence there is no guarantee that the dissipation anomaly for the constructed sequence of solutions would be observable in numerical experiments run over long time. The existence of such a sequence  remains an important open problem in mathematical theory of turbulence.

\begin{remark}
Changing the time scale to (almost) exponential $t_n = 1-\lambda_n^{-1} n^2$ in the construction ensures that the the sequence $u^{\nu_j}$ of solutions to \eqref{eq:NSE} in Theorem~\ref{thm:DA_in_FT} is bounded in $L^3(0,T;C^\alpha)$ for every $\alpha <1/3$ and $T>0$, i.e., it is almost Onsager critical. In addition, it is possible to make it intermittent and hence even more physically relevant. These more complex constructions are out of scope of this note and will be discussed in a follow up paper.
\end{remark}

\begin{remark}
The integral length scale in Theorems~\ref{thm:DF} and \ref{thm:DA_in_FT} for simplicity is chosen to be the size of the torus. This makes bound \eqref{bound_on_drag_coefficient_thm} essentially as originally written by Foias \cite{MR1467006}. We have shown that there are shapes of forces that result in arbitrary large energy dissipation normalized such a way.
However, for the construction of Theorem~\ref{thm:DA_in_FT}, it is more natural to choose the integral length scale to be
\[
\mathcal{L}= U \tau,
\]
where $\tau$ is the time period of the solutions (see the proof of the theorem). Then
\[
C_{\epsilon} := \lim_{j\to \infty} \frac{\epsilon \mathcal{L} }{U^3} = 2,
\]
for the construction in the above theorem, which is consistent with experimental data collected in \cite{10.1063/1.2055529}. We refine upper bound \eqref{bound_on_drag_coefficient_thm} exploiting a space-time frequency support of a force in a companion paper.
\end{remark}

\begin{remark} Under long-time averaging, the dissipation anomaly is only possible when there is an injection of the energy. 
In fact, the average energy dissipation is equal to the average work done by the force, i.e.,
\[
\epsilon = \underset{T\to \infty}{\mathrm{Lim}} \frac{1}{T}\int_0^T (u^{\nu},f^{\nu}) \, dt,
\]
for smooth solutions $u^\nu$ (less than or equal for Leray-Hopf solutions). In Theorem~\ref{thm:DA_in_FT} we also have
\[
\lim_{\nu \to 0} \epsilon = \lim_{T\to \infty} \frac{1}{T}\int_0^T (u,f) \, dt,
\]
where $u$ is the limiting solution of the Euler equation such that $u^{\nu_j} \to u$ in $C_w(\RR; L^2(\TT^3))$.
\end{remark}

\vspace{0.1in}

\noindent
{\bf Fractal-forced turbulence.}
To emphasize the $L^2$ criticality of the force for the energy dissipation, we note that 
experimental and computational studies of so called fractal-generated turbulence have been performed in the past two decade, see \cite{gomes-fernandes_ganapathisubramani_vassilicos_2012} and references therein. In such flows the energy is injected at a wide range of scales in a self-similar fashion. In our setup this would correspond to the convergence in $H^{-\alpha}$ of the shapes of the force $\phi^j$ for some $\alpha>0$ (but not in $L^2$). As shown in \cite{MR2337007} (where the proof was done in the case of one fixed force shape), bound \eqref{b:bound_on_drag_coefficient} elevates to 
\[
\frac{\epsilon \ell}{U^3} \leq c_1 Re^{\frac{\alpha}{2-\alpha}}
 + c_2 Re^{-1}.
\]
The existence of flows that attain this bound, even in the case of time-dependent forces, is not known. However, there is numerical and experimental evidence confirming that the energy dissipation does not exhibit Kolmogorov scaling $\frac{\epsilon \ell}{U^3} \sim Re^0$, but rather $\frac{\epsilon \ell}{U^3} \sim Re^s$ for $s>0$ depending on the ``roughness'' of the force (i.e., the value of $\alpha$).

\subsection{Previous results} \label{Subsec:Previous_results}

\subsubsection{Onsager's conjecture and anomalous dissipation} From the positive side, starting with the work of Eyink \cite{MR1302409}, the exact H\"older regularity exponent conjectured by Onsager \cite{Onsager} has been reached by Constantin, E, and Titi in  \cite{MR1298949} with refinements  in \cite{MR1734632,MR2422377}, where the latter work shows that the energy balance holds for weak solutions in $L^3_tB^{1/3}_{3,c_0}$. Sharp lower bounds on the Hausdorff dimension of the space-time set which can support anomalous dissipation were obtained in \cite{derosa2023support}.

From the negative side, the first constructions of non-conservative solutions go back to the works of Scheffer \cite{MR1231007} and Shnirelman \cite{MR1476315}. More regular constructions appeared with a help of convex integration, a technique originated in isometric embedding problems of geometry dating back to the work of Nash and Kuiper. Its first application to fluid dynamics was carried out in 2009 by the pioneering work of De Lellis and Székelyhidi Jr. \cite{MR2600877}. Since then, this technique has been very fruitful in the fluids community. Remarkably, its development over a series of works has culminated in the resolution of the Onsager conjecture for the 3D Euler equations by Isett \cite{MR3866888} with a construction of a non-conservative weak solution in $C_{x,t}^{1/3-}$. In the scale of $L^2$-based Sobolev spaces, recently the authors in \cite{MR4649134, MR4601999} were able to show anomalous dissipation in $C_t H^{1/2-}$. In \cite{giri2023l3based}, anomalous dissipation with local energy inequality was obtain in $C_t W^{1/3-,3}$. Convex integration was even used to approach the Onsager space in two dimensions \cite{giri20232d}, where physically acceptable solutions arising in the limit of vanishing viscosity satisfy the energy balance even outside of the Onsager class \cite{MR3551263, MR4228012}. Two-dimensional turbulent flows are characterized by the inverse energy cascade and formation of large vortex structures, which present obstacles for the anomalous energy dissipation in the limit of vanishing viscosity. Convex integration seems to allow bypassing such obstacles. Remarkably, convex integration can be used to construct weak solutions of the Navier-Stokes equations testing the sharpness of positive uniqueness results \cite{MR3898708, MR4422213, MR4462623, MR4610908}, but without Kolmogorov's dissipation range.
Inviscid turbulent fluid flows constructed via convex integration exhibit some features of turbulence, but they are not vanishing viscosity limits of physical Navier-Stokes solutions with Kolmogorov's dissipation range.
Also, the locality of interactions in frequency space is a missing ingredient in current schemes. The locality of interactions is a feature of turbulent flows, and is necessary to reach the exact physical Onsager's $1/3$ exponent, see \cite{MR2422377}. Constructions based on perfect mixing, as studied in this paper, can produce solutions to the forced 3D Euler equations in $L^3_tC^{1/3-}_t$, see \cite{MR4595604}. We address the possibility of reaching the exact Onsager's $1/3$ exponent in a companion paper.	

\subsubsection{Effect of intermittency on the energy flux} While the energy spectrum of spatially homogeneous fields cannot decay faster than $\kappa^{-5/3}$ to sustain a non-vanishing in frequency energy flux, it can decay as fast as $\kappa^{-8/3}$ for highly intermittent fields and still produce energy cascade (see the volumetric theory of intermittency \cite{MR4581460} for precise definitions). Indeed, the $L^2$-based Onsager exponent in 3D is $5/6$, see \cite{MR2665030}. The fact that the nonlinear term becomes stronger for intermittent flows has been exploited in \cite{MR2566571, MR4283701} where Dirichlet-type building blocks with extreme intermittency were used to produce Leray-Hopf solutions of the Navier-Stokes equations discontinuous in the largest regularity-critical space $B^{-1}_{\infty,\infty}$, and discontinuous in the energy space $L^2$ while staying in borderline Onsager $L^p_tL^q_x$, $\frac2p+ \frac2q=1+$ spaces. Leray-Hopf solutions with low intermittency dimension stay regular, see \cite{MR3208714} for the intermittency threshold. In fact, the use of intermittent building blocks by Buckmaster and Vicol was essential in their groundbreaking proof of non-uniqueness for the 3D Navier-Stokes equations \cite{MR3898708}. The use of intermittency was essential in \cite{MR4649134, MR4601999} as well, even though current convex integration schemes seem not to be compatible with the extreme intermittency needed to reach Onsager critical $L^2$-based $H^{5/6}$ space. In this paper we follow the framework of Bru\`e and De Lellis \cite{MR4595604} based on spatially homogeneous mixing, which can be rescaled in time to reach the $L^\infty$-based Onsager space $L^3_tC^{\alpha}_x$, $\alpha<1/3$, but we stick to the original framework here for simplicity. In a companion work we introduce some intermittency and compute the multifractal spectrum, often measured in experiments and numerical simulations,  for such solutions.

\subsubsection{Dissipation Anomaly}
While dissipation anomaly is not possible in two-dimensional flows \cite{MR3187680} due to the absence of forward energy cascade, the known upper bounds on the energy dissipation in three dimensions are consistent with numerical and experimental observations.
Rigorous upper bounds on the energy dissipation rate for boundary driven flows go back to the work by Doering and Constantin \cite{PhysRevLett.69.1648}.
The first rigorous upper bound on the normalized dissipation rate for the forced 3D Navier-Stokes equations were obtained by Foias in \cite{MR1467006}, which was later refined as bound \eqref{b:bound_on_drag_coefficient-intro} by Foias and Doering in \cite{MR1928940}.

For the dyadic model of the Navier-Stokes equations, the author and Friedlander in \cite{MR2522972} proved that the the long time average of the energy dissipation does not vanish in the limit of vanishing viscosity for all the solutions. So this model exhibits dissipation anomaly not only for the the long time, but also ensemble averages, which is currently out of reach for the real equations. 
The first construction with dissipation anomaly for the advection-diffusion equation was obtained by Drivas, Elgindi, Iyer, and Jeong in \cite{MR4381138}.
In a very recent work \cite{armstrong2023anomalous}, Armstrong and Vicol constructed a divergence free vector field that produces dissipation anomaly for every $H^1$ initial scalar.
Also, in \cite{MR4662772}, Colombo, Crippa, and Sorella have shown dissipation anomaly for the advection-diffusion equation with any chosen regularity in the supercritical  Obukhov-Corrsin regime, as well as the lack of the selection by vanishing diffusivity, proving that there are two distinct limit solutions of the transport equation. In this paper we prove the existence of {\em infinitely} (but countably) many solution of the Euler equations in the limit of vanishing viscosity, see Theorem~\ref{thm:main-2}. It is unknown whether this holds for the advection diffusion equation as well. In \cite{2310.02934}, Burczak, Sz\' ekelyhidi, and Wu used convex integration to construct a solution to the Euler equations, such that dissipation anomaly holds for solutions of the advection diffusion equation with any initial data in $H^1$.

Bru\`e and De Lellis \cite{MR4595604} recently constructed the first example of   dissipation anomaly for the Navier-Stokes equations in the sense of Definition~\ref{Def:Dissipation_Anomaly_Intro}. In their scheme a two-dimensional vector field $v^\nu$ is efficiently mixing a scalar $\theta^\nu$, which plays a role of the third component of the constructed solution $u^\nu$ to the \eqref{eq:NSE}. The $2+\frac12$ dimensional flows produce instabilities in two-dimensional flows, see Yudovich \cite{MR1791984}, and can be used to show the enhance dissipation for the 3D Navier-Stokes equations, see Jeong and Yoneda  \cite{MR4297205,MR4375721}.
The mixing constructions of non-unique solutions to the transport equation go back to Aizenman  \cite{Aizenman1978}; see also Depauw \cite{MR2009116}.
In \cite{MR4595604}, the authors use a smooth mixing construction by  Alberti, Crippa, and Mazzucato \cite{MR3904158}. In \cite{BrueEtc}, the authors show dissipation anomaly for a sequence of solutions to the 3D NSE bounded in $L^3_tC^{\alpha}_x$, $\alpha<1/3$. In  \cite{johansson2023nontrivial},  an example of absolutely continuous dissipation anomaly in four dimensions was obtained.  In three dimensions, an example of the absolutely continuous dissipation anomaly is given by Theorem~\ref{thm:main} (see Remark~\ref{rem:AbsContin}).

\section{Main results}

Here we state our main results, including examples of the following phenomena:
\begin{itemize}
\item Total (and partial) loss of energy due to the dissipation anomaly.
\item Dissipation anomaly and anomalous dissipation with absolutely continuous energy limit $E(t)$.
\item Anomalous dissipation without dissipation anomaly.
\item The existence of infinitely many limiting solutions of the Euler equation in the limit of vanishing viscosity.
\item  Dissipation anomaly implies the discontinuity of $\|u(t)\|_{L^2}$ for  quasi-selfsimilar constructions.

\end{itemize}

\begin{theorem} \label{thm:main}
There is a countable family of smooth solutions to \eqref{eq:NSE} $\{u^{\nu}(t)\}_{\nu}$ on time interval $[0,2]$ with force $f^{\nu} \to f$ in $C([0,2];C^\alpha)$ for all $0<\alpha<1$  as $\nu \to 0$, initial data
$u^{\nu}(0) = u_{\mathrm{in}}$ and satisfying the following.

There exists a solution of the Euler equation $u \in C_w([0,2];L^2)$, smooth on $[0,1)\cup(1,2]$, with force $f$, initial data $u(0) =u_{\mathrm{in}}$, such that 
\begin{equation} \label{eq:anom_dis_FirstMainThm}
\int_0^1 (f,u) \, dt =0, \qquad \|u(1)\|_{L^2}^2 = \|u_{\mathrm{in}}\|_{L^2}^2=1, \qquad u(t) = 0  \ \ {\text for} \ \  t \in[1,2],
\end{equation}
and as $\nu \to 0$, the family of the NSE solutions  $u^{\nu}$
converges strongly on $[0,1)$:
\[
u^{\nu} \to u \qquad \text{in} \qquad C([0,t];L^2), \qquad \forall t \in [0,1).
\]

Moreover, the family $\{u^{\nu}(t)\}_{\nu}$ contains the following sequences.

\noindent
{\bf First subfamily with total dissipation anomaly on $[0,1+]$:} 
For any energy level $e\in[0,1]$
there exists a subsequence $\nu^{e}_j \to 0$ as $j\to \infty$ such that $u^{\nu^e_j} \to u$ in $C_w([0,2];L^2)$ and $u^{\nu^e_j}$ dissipates this amount of energy on $[0,1]$  in the limit of vanishing viscosity:
\usetagform{partial_total_or_no}
\begin{equation} \label{Thm1_DA_1}
\lim_{j\to \infty} 2\nu^e_j \int_{0}^{1} \| \nabla u^{\nu^e_j} \|_{L^2}^2 \, dt=e.
\end{equation}
\usetagform{default}
On the other hand, $u^{\nu^e_m}$ dissipates the total energy on any larger interval:
\usetagform{total}
\begin{equation} \label{Thm1_DA_2}
\lim_{j\to \infty} 2\nu_j^e \int_{0}^{t} \| \nabla u^{\nu^e_j} \|_{L^2}^2 \, dt=1, \qquad \forall t\in (1,2].
\end{equation}
\usetagform{default}
In particular, the limiting energy is discontinuous:
\[
E(t)= \lim_{j\to \infty} \|u^{\nu^e_j}(t)\|_{L^2}^2 =
\begin{cases}
\|u(t)\|_{L^2}^2, & t\in[0,1),\\
0, & t\in[1,2].
\end{cases}
\]

\noindent
{\bf Second subfamily with partial dissipation anomaly on $[0,2]$ and absolutely continuous limiting energy $E(t)$:} For any energy level $e\in[0,1)$
there exists a subsequence $\nu^e_j \to 0$ as $j\to \infty$, such that $u^{\nu^e_j} \to u$ in $C_w([0,2];L^2)$ and $u^{\nu^e_j}$ dissipates this amount of energy on $[0,2]$  in the limit of vanishing viscosity:
\usetagform{partial_or_no}
\begin{equation} \label{Thm1_DA_3}
\lim_{j\to \infty} 2\nu^e_j \int_{0}^{2} \| \nabla u^{\nu^e_j} \|_{L^2}^2 \, dt=e.
\end{equation}
\usetagform{default}
Moreover, the limiting energy $E(t)$ is positive and absolutely continuous on $[0,2]$:
\[
E(t):= \lim_{j\to \infty} \|u^{\nu^e_j}(t)\|_{L^2}^2 =
\begin{cases}
\|u(t)\|_{L^2}^2, & t\in[0,1),\\
{\displaystyle \lim_{\tau \to 1-}} \|u(\tau)\|_{L^2}^2 = 1, & t=1,\\ 
\text{absolutely continuous, decreasing},& t\in[1,2],\\
1-e, & t=2.
\end{cases}
\]

In particular, 
\[
\lim_{j \to \infty} \|u^{\nu^e_j}(t) \|_{L^2}^2 \geq 1-e >0 =  \|u(t) \|_{L^2}^2, \qquad t\in[1,2],
\]
and hence $u^{\nu^e_j}(t)$ does not converge strongly in $L^2$ to $u(t)$ for every $t\in[1,2]$.

Also, when $e=0$, there is no {\bf dissipation anomaly} by \eqref{Thm1_DA_3}, while the limiting solution of the Euler equation looses all of its energy exhibiting {\bf anomalous dissipation}   \eqref{eq:anom_dis_FirstMainThm}.
\end{theorem}

\begin{remark} In fact, the family of the NSE solutions $u^\nu(t)$ from the above theorem exhibits the total dissipation anomaly on any time interval $[t_1,t_2]$ with $t=1$ in the interior, $t_1<1<t_2$:
\[
\lim_{j\to \infty} 2\nu^e_j \int_{t_1}^{t_2} \| \nabla u^{\nu^e_j} \|_{L^2}^2 \, dt=1, \qquad \forall 0 \leq t_1<1<t_2 \leq 2.
\]
Some dissipation may occur before $t=1$ and some after $t=1$.
In particular, for $e=0$,
\[
\lim_{j\to \infty} 2\nu^e_j \int_{1}^{t_2} \| \nabla u^{\nu^e_j} \|_{L^2}^2 \, dt=1, \qquad \lim_{j \to \infty} \|u^{\nu^e_j}(1)\|_{L^2}^2 =\lim_{t\to 1-}\|u(t)\|_{L^2}^2= 1, \qquad \text{for} \qquad e=0,
\]
i.e., there is no dissipation anomaly before $t=1$ in this case. All the dissipation anomaly occurs after $t=1$.

In the other extreme case $e=1$,
\[
\lim_{j\to \infty} 2\nu^e_j \int_{t_1}^{1} \| \nabla u^{\nu^e_j} \|_{L^2}^2 \, dt=1, \qquad\lim_{j \to \infty} \|u^{\nu^e_j}(1)\|_{L^2}^2 = \|u(1)\|_{L^2}^2 = 0, \qquad \text{for} \qquad e=1,
\]
as  all the available energy dissipates before $t=1$ in this case. The dissipation anomaly is the largest possible on the time interval $[0,1]$, and there is no energy left to dissipate after $t=1$.

Finally, for the second subfamily with partial dissipation anomaly on $[0,2]$ and continuous limiting energy $E(t)$, the dissipation anomaly occurs continuously on the time interval $[1,2]$, where the limiting solution $u(t)$ is zero:
\[
D(t):= \lim_{j \to \infty} 2\nu\int_{0}^t \|\nabla u^{\nu_j^e}(\tau)\|_{L^2}^2 \, d\tau =
\begin{cases}
0, & t\in[0,1],\\
\text{absolutely continuous, increasing},& t\in[1,2],\\
e, & t=2.
\end{cases}
\]

These scenarios are consistent with Cases I - III in the proof of Theorem~\ref{thm:Discont_of_u} (on discontinuity of $u(t)$).
\end{remark}

\begin{remark}
In the above construction on the second half of the time interval
\[
u(t)=0, \qquad u^\nu(t)=(0,0,\theta^\nu(t)), \qquad f^{\nu}(t) = 0,  \qquad t\in[1,2],
\] 
where $\theta^\nu$ satisfies the heat equation
\[
\p_t \theta^\nu  = \nu \Delta \theta^\nu, \qquad \theta^\nu(1) = u^{\nu}_3(1).
\]
\end{remark}

\begin{remark} \label{rem:AbsContin}
Since the limiting energy $E(t)$ is absolutely continuous for the second subfamily in the above theorem, the limit of the dissipation measure is absolutely continuous with respect to the Lebesgue measure.
\end{remark}

Another way to extend the solutions $u^\nu(t)$ of the NSE beyond $t=1$ is to produce an inverse energy cascade via reversing the solution of the transport equation $\rho$ and the velocity $v$ in time.
The following theorem shows that the inverse energy cascade (as $t \to 1+)$ can reduce the loss of energy as solutions cross $t=1$. Moreover, it can produce infinitely many limiting solutions of the Euler equations.

\begin{theorem} \label{thm:main-2}
There is a (countable)  family of smooth solutions to the 3D NSE \eqref{eq:NSE} $\{u^{\nu}(t)\}_{\nu}$ on time interval $[0,2]$ with force $f^{\nu} \to f$ in $C([0,2];C^\alpha)$ for all $0<\alpha<1$ as $\nu \to 0$, initial data
$u^{\nu}(0) = u_{\mathrm{in}}$,
and satisfying the following.

There exist two weak solutions of the Euler equation $u_1, u_2 \in C_w([0,2];L^2)$ smooth on $[0,1)$ and $(1,2]$ with force $f$, initial data $u_1(0)=u_2(0) =u_{\mathrm{in}}$, such that
\begin{enumerate}
\item
$ u_1(t) = u_2(t), \qquad \forall t \in[0,1]$,
\item
$\|u_1(t)\|_{L^2}^2> \|u_2(t)\|_{L^2}^2, \qquad \forall t\in(1,2]$
\item
$\displaystyle \lim_{t\to 1} \|u_1(t)\|_{L^2}^2=\lim_{t\to 1-} \|u_2(t)\|_{L^2}^2>0, \qquad \lim_{t\to 1+} \|u_2(t)\|_{L^2}^2=0, \qquad u_1(1) = u_2(1) = 0,$
\item \label{eq:4th_item_them_main-2}
$u_1(t)$ satisfies the energy equality everywhere except $t=1$:
\[
\|u_1(t)\|_{L^2}^2 =  \|u_1(0)\|_{L^2}^2  + 2\int_{0}^{t} (f,u_1) \, d\tau, \qquad \forall t\in[0,1)\cup(1,2].
\]
\end{enumerate}

\noindent
{\bf Two extreme limiting solutions of the Euler equation:} 
There exist two subsequences  $\nu^1_j \to 0$ and $\nu^2_j \to 0$ as $j\to \infty$ such that $u^{\nu^1_j}$ converges to $u_1$ and $u^{\nu^2_j}$ converges to $u_2$ in the following sense:
\[
u^{\nu^1_j} \to u_1, \qquad u^{\nu^2_j} \to u_2 \qquad \text{in} \qquad C_w([0,2];L^2),
\]
\[
u^{\nu^1_j} \to u_1, \qquad u^{\nu^2_j} \to u_2 \qquad \text{in} \quad  C([0,t];L^2), \ t\in[0,1), \quad \text{and in} \quad C([t,2];L^2), \ t\in(1,2),
\]
Moreover,
\[
u^{\nu^1_j}(1) \nrightarrow u_1(1) \quad \text{in} \quad L^2 , \qquad \text{while} \qquad u^{\nu^2_j} \to u_2, \quad \text{in} \quad C([1,2];L^2).
\]
In particular, the sequence $u^{\nu^1_j}$  does not exhibit the dissipation anomaly (see \eqref{eq:4th_item_them_main-2}) while $u^{\nu^2_j}$ does: 
\usetagform{noDA}
\begin{equation}
\lim_{j\to \infty} 2\nu^1_j \int_{0}^{2} \| \nabla u^{\nu^1_j} \|_{L^2}^2 \, dt=0,
\end{equation}
\usetagform{total}
\begin{equation}
\qquad \lim_{j\to \infty} 2\nu^2_j \int_{0}^{2} \| \nabla u^{\nu^2_j} \|_{L^2}^2 \, dt= \|u_{\mathrm{in}}\|_{L^2}^2 =1.
\end{equation}

\noindent
{\bf Arbitrary dissipation anomaly on $[0,2]$ and infinitely many limiting solutions of the Euler equation:} 
\usetagform{partial_total_or_no}
For any $e\in[0,1]$ there exists a subsequence  $\nu^e_j \to 0$ as $j \to \infty$ with 
\begin{equation}
 \lim_{j\to \infty} 2\nu^e_j \int_{0}^{2} \| \nabla u^{\nu^e_j} \|_{L^2}^2 \, dt=e.
\end{equation}
\usetagform{default}
Moreover, there exist infinitely many solutions of the Euler equation $u_n(t)$, $n=3,4, \dots$ with $u_n(0)=u_{\mathrm{in}}$ coinciding with $u_1(t)$ and $u_2(t)$ on $[0,1]$ and satisfying
\[
\|u_n(2)\|^2_{L^2} <  \|u_{\mathrm{in}}\|_{L^2}^2 =1, \qquad n=3,4,\dots,	
\]
and
\[
\lim_{n\to \infty} \|u_n(2)\|_{L^2}^2 = 1.
\]
Finally, each $u_n(t)$ is attained in the limit of vanishing viscosity, i.e., for every $n\in \mathbb{N}$,
\[
u^{\nu^n_j} \to u_n, \qquad \text{in} \qquad C_w([0,2];L^2),
\]
as $j \to \infty$, for some subsequence  $\nu^n_j \to 0$.
\end{theorem}

\noindent
{\bf Open question.}
Does $u^{\nu^n_j}$ converges to $u_n$ strongly in $L^2$ for $n\geq 3$?

\bigskip

Finally, we prove that the dissipation anomaly implies the discontinuity of the limit solution $u(t)$ in $L^2$
provided a certain frequency localization property is satisfied, reminiscent of Tao's delay mechanism in \cite{MR3486169}. 
Such a frequency localization property is enjoyed by quasi-selfsimilar constructions.  This conditionally answers Question~(\ref{Q:Cont_u}) in the Introduction. 

\begin{theorem} \label{thm:discont-intro}
Let $u^{\nu}(t)$ be a family of weak solutions to \eqref{eq:NSE} satisfying the energy equality with viscosity $\nu \to 0$ and force $f^{\nu} \to f$ in $L^1(0,1;L^2)$, converging weakly in $L^2$ to $u \in L^\infty(0,1;L^2)$
\[
u^{\nu} \to u \qquad \text{in} \qquad C_w([0,1];L^2),
\]
converging strongly at $t=0$
\[
u^{\nu}(0) \to u(0) \qquad \text{in} \qquad L^2,
\]
and exhibiting the dissipation anomaly, i.e.,
\begin{equation} \label{eq:diss_anomaly}
\limsup_{\nu \to 0} \nu \int_0^1 \| \nabla u^{\nu} \|_{L^2}^2 \, dt >0.
\end{equation}
Assume also that there are constants $c>0$, $\alpha>1$ such that for every $m\in \NN$ and $t\in[0,1]$ there exists $\tilde{q}(\nu,t)$ with the following localization  property:
\begin{equation} \label{eq:kernel-theorem}
\|\Delta_q u^{\nu}(t)\|_{L^2} \leq c \lambda_{|q-\tilde q(\nu,t)|}^{-\alpha}, 
\end{equation}
for all $q \in \NN$ and all $\nu$, where $\Delta_qu^\nu$ is the Littlewood-Pale projection of $u^\nu$ onto frequencies of size $\sim \lambda_q:=2^q$. Then $u(t)$ is discontinuous in $L^2$ at some $t\in[0,1]$.
\end{theorem}

\begin{remark}
In fact, the above result also holds for Leray-Hopf solutions, i.e., weak solutions satisfying the energy inequality starting from almost all initial data.
\end{remark}

\begin{remark}
As we show in Theorem~\ref{thm:Discont_of_u}, the localization assumption can be weakened as follows:
\[
u^{\nu}=u^\nu_{\mathrm{Loc}} + u^\nu_{\mathrm{Ons}},
\]
where $u^\nu_{\mathrm{Loc}}$ satisfies \eqref{eq:kernel-theorem} and $u^\nu_{\mathrm{Ons}}$ converges to a function in the Onsager space and does not exhibit dissipation anomaly, see Assumption~\ref{assumption}.
\end{remark}

\section{Perfect mixing} \label{sec:Mixing}

For the constructions, as in  \cite{MR4595604}, we will take advantage of $2+\frac12$ dimensional solutions of the Euler equations
\[
u(x,t)=(v, \rho),
\]
where a 2D divergence free vector field $v$ efficiently mixes a scalar $\rho$ satisfying the transport equation
\[
\p_t \rho + v \cdot \nabla \rho =0. 
\]
More precisely, we will be using the smooth mixing construction by Alberti, Crippa, and Mazzucato \cite{MR3904158},
where $\rho \in C^\infty([0,1) \times \TT^2)$, $v \in C^\infty([0,1) \times \TT^2; \RR^2)$, and
\[
\rho(t) \rightharpoonup 0 \quad \text{weakly in } L^2 \text{ as }  t \to 1-.
\]
Moreover, $ v \in L^\infty(0,1; C^{\alpha})$ for all $\alpha \in (0,1)$. Denote the frequency $\l_n = 5^n$.
The following result was proved in \cite{MR3904158} (see Theorem 4.1 in \cite{MR4595604}). 

\begin{theorem} \label{thm:v_n_rho_n}
There exist smooth solutions to the transport equation $\rho_n \in C^\infty([0,1]\times \TT^2 )$ with smooth drift  $v_n \in C^\infty([0,1] \times \TT^2 ; \RR^2)$, such that, for every $n\in \NN$,
\begin{enumerate}
\item $\|\p_t^k v_n\|_{L^\infty(0,1;C^\alpha(\TT^2))} \leq C(\alpha,k) \l_n^{\alpha-1}$, for every $\alpha \geq 0$ and $k \in \NN$.
\item $\rho_n(t)$ has zero mean and $\|\rho_n(t)\|_{L^2} =1$ for every $t\in[0,1]$, and
\begin{equation} \label{eq:building_blocks_estimates}
\|\rho_n(t)\|_{L^\infty} \leq 10, \qquad \|\nabla \rho_n(t)\|_{L^\infty} \leq C \l_n, \qquad \|\rho_n(t)\|_{\dot H^{-1}} \leq C\l_n^{-1},
\end{equation}
for some absolute constant $C$.
\item $\rho_n(1)=\rho_{n+1}(0)$ for every $n \in \NN$.
\end{enumerate}
\end{theorem} 
In all the constructions the smooth initial data for solutions of \eqref{eq:NSE} and the Euler equations will have $\rho_1(0)$ as the third component:
\begin{equation} \label{eq:ID}
u_{\mathrm{in}} := (0,0, \rho_{\mathrm{in}}), \qquad \rho_{\mathrm{in}} := \rho_1(0) \in C^{\infty}, \qquad \|u_{\mathrm{in}} \|_{L^2}=\|\rho_{\mathrm{in}} \|_{L^2} =1.
\end{equation}

Following  \cite{MR4595604}, the rescaled velocities and densities are glued in time to obtain smooth solutions of the transport equation on $[0, 1]$
\begin{equation} \label{eq:def_tilde_v}
\tilde v^{m}(x,t)=\sum_{n=0}^{m} \eta'(t) \chi_{[t_n,t_{n+1})}(\eta(t)) \frac{1}{t_{n+1}-t_n}v_n\left(x,\frac{\eta(t)-t_n}{t_{n+1}-t_n}\right),
\end{equation}
\begin{equation} \label{eq:def_tilde_rho}
\tilde \rho^{m}(x,t)= \sum_{n=0}^{m}  \chi_{[t_n,t_{n+1})}(\eta(t)) \rho_n\left(x,\frac{\eta(t)-t_n}{t_{n+1}-t_n}\right) + \chi_{[t_{m+1},1]} \rho_m(x,1),
\end{equation}
where $t_n=1-(n+1)^{-2}$ and $\eta : [0, 1] \to [0, 1]$ is some smooth non-decreasing function satisfying
\begin{enumerate}
\item $\eta(t_n)=t_n$ for any $n \in \NN$;
\item $\eta^{(k)} (t_n)=0$ for any $n,k \in \NN$, $k \geq 1$;
\item $|\eta^{(k)}(t)| \chi_{[t_n, t_{n+1})} \leq C(k) n^{5k}$ for any $n,k\in \NN$, and $t\in[0,1]$.
\end{enumerate}
Clearly $\tilde \rho^m$ enjoys the same estimates \eqref{eq:building_blocks_estimates} as the last building block in the sum 
\begin{equation} \label{eq:estimates_rho}
\| \tilde \rho^m(t)\|_{L^\infty} \leq 10, \qquad \|\nabla \tilde\rho^m(t)\|_{L^\infty} \leq C \l_m, \qquad \|\tilde\rho^m(t)\|_{\dot H^{-1}} \leq C\l_m^{-1},
\end{equation}
for some absolute constant $C$.

The limits as $m \to \infty$ of $(\tilde v^{m}, \tilde \rho^{m})$ provides an {\em explicit} solution of the Euler equations that looses all the energy at time $t=1$.  More precisely, define
\begin{equation} \label{eq:def_tilde_v-limit}
\tilde v(x,t)=\sum_{n=0}^{\infty} \eta'(t) \chi_{[t_n,t_{n+1})}(\eta(t)) \frac{1}{t_{n+1}-t_n}v_n\left(x,\frac{\eta(t)-t_n}{t_{n+1}-t_n}\right),
\end{equation}
\begin{equation} \label{eq:def_tilde_rho-limit}
\tilde \rho(x,t)= \sum_{n=0}^{\infty}  \chi_{[t_n,t_{n+1})}(\eta(t)) \rho_n\left(x,\frac{\eta(t)-t_n}{t_{n+1}-t_n}\right),
\end{equation}
which is a smooth solution to the transport equation on $[0,1)$:
\begin{equation} 
\p_t \tilde \rho + \tilde v \cdot \nabla \tilde \rho =0,
\end{equation}
with the initial data $\tilde \rho(0)= \rho_{\mathrm{in}}$ defined in \eqref{eq:ID}. Now define a $2+\frac12$ dimensional velocity
\begin{equation} \label{eq:Def_u}
u(x,t)=(\tilde v(x_1,x_2,t),\tilde \rho(x_1,x_2,t)).
\end{equation}
It is easy to see that $u(t)$ is a smooth solution of the Euler equations on $[0,1)$ with $u(0)=u_{\mathrm{in}}$ defined in \eqref{eq:ID} and the force $f =(g,0)$, where
\begin{equation} 
g:= \p_t \tilde v +(\tilde v \cdot \nabla)\tilde v.
\end{equation}
Note that $u(t)$ blows-up as $t \to 1-$ as
\begin{equation} \label{eq:u_weak_limit}
\lim_{t \to 1-} \|u(t)\|_{L^2} = 1, \qquad u(t) \rightharpoonup 0 \text{ weakly in }  L^2  \text{ as }  t \to 1-.
\end{equation}
We will later extend $u(t)$ to a weak solution on $[0,2]$ in two different ways explicity, and, among other things, prove the existence of infinitely many other solutions of the Euler equations with the same initial data and force (see Theorem~\ref{thm:main-2}).

The main goal of this paper is to study the behavior of solutions to the \eqref{eq:NSE} with initial data $u_{\mathrm{in}}$ by analyzing the evolution of $\rho_{\mathrm{in}}$ under the advection-diffusion equation in the limit of vanishing viscosity. We consider solutions to the \eqref{eq:NSE} that also enjoy the $2+\frac12$ dimensional structure
\[
u^{m}(x,t)=(\tilde v^{m}(x_1,x_2,t),\theta^m(x_1,x_2,t)),
\]
where $\theta^m$ satisfies the advection diffusion equation
\begin{equation} 
\p_t \theta^m + \tilde v^m \cdot \nabla \theta^m = \nu_m \Delta \theta^m,
\end{equation}
with $\theta^m(0)=\rho_{\mathrm{in}}$. Indeed, $u^{m}(t)$ satisfies \eqref{eq:NSE} with force
\begin{equation} \label{eq:def_g^m}
g^m:= \p_t \tilde v^m +(\tilde v^m \cdot \nabla)\tilde v^m - \nu_m \Delta \tilde v^m.
\end{equation}
A direct computation (see Lemma~5.1 in \cite{MR4595604} with a particular choice $\nu_m =m^{10}\l_m^{-2}$) shows that for any sequence $\nu_m \lesssim m^s\l_m^{-2}$ for some $s \in \RR$, the NSE force $g^m$ converses to the Euler force $g$ in 
in H\"older spaces:
\[
g^m \to g, \qquad \text{in} \quad C([0,1];C^\alpha(\TT^2)), \quad \forall \alpha \in (0,1).
\]
In particular, such a force does not affect the strength of the energy cascade or the dissipation anomaly mechanism as there is no anomalous work done by the force on high modes (see the discussion in Introduction).

We will use the following simple lemmas. The first one ensures that the work of the limiting force is zero, and the energy loss due to the dissipation anomaly for the solutions of the NSE and advection-diffusion equations are the same in all the constructions.

\begin{lemma} \label{l:DA_for_u=DA_for_th}
Let $u^\nu=(v^\nu,\theta^{\nu})$ be a family of solutions to the 3D NSE \eqref{eq:NSE} on $[0,2]$ with initial data
$u^{\nu}(0) = u_{\mathrm{in}}$,  $v^\nu \in C^\infty([0,2] \times \TT^2; \RR^2)$, $\theta^\nu \in C^\infty([0,2] \times \TT^2 )$, force $f^{\nu} \to f$ in $L^1(0,2;L^2(\TT^3))$, and such that  $\{v^\nu\}_\nu$ is bounded in $L^2(0,2;H^1)$ and  $v^\nu(t) =0$ for all $m$ and $t \in \mathcal{S} \subset [0,2]$. Then

\begin{equation} \label{eq:DA_for_u=DA_for_th}
\lim_{\nu\to 0} \left( \nu \int_{0}^{t} \| \nabla u^{\nu} \|_{L^2}^2 \, d\tau -  \nu \int_{0}^{t} \| \nabla \theta^{\nu} \|_{L^2}^2 \, d\tau \right)=0, \qquad \forall t \in \mathcal{S},
\end{equation}
and
\begin{equation} \label{eq:Zero_Work}
\int_{0}^{t}  (f,u) \, d\tau =0, \qquad \forall t \in \mathcal{S}.
\end{equation}
\end{lemma}
\begin{proof}
Since $\{v^\nu\}_\nu$ is bounded in $L^2(0,2;H^1)$, we have
\begin{equation} \label{eq:v-NoDA}
\lim_{\nu\to 0} \nu \int_{0}^{t} \| \nabla v^{\nu} \|_{L^2}^2 \, d\tau=0, \qquad \forall t\in [0,2],
\end{equation}
which implies \eqref{eq:DA_for_u=DA_for_th}.

Now thanks to the fact that $v^\nu(t) =0$ for all $m$ and $t \in \mathcal{S}$, we have
\begin{equation} \label{eq:u=theta_Therem1_proof}
\|u^\nu(t)\|_{L^2} = \|\theta^\nu(t)\|_{L^2}, \qquad t \in \mathcal{S}.
\end{equation}
In addition, $u^\nu(0) = (0,0,\theta^\nu(0))= u_{\mathrm{in}}$ for all $\nu$. Hence, subtracting the energy equalities for $u^\nu$ and $\theta^\nu$ and using \eqref{eq:DA_for_u=DA_for_th} we get
\[
\begin{split}
\int_{0}^{t} (f, u) \, d\tau &= \lim_{\nu \to 0} \int_{0}^{t} (f^\nu, u^\nu) \, d \tau\\
&=\lim_{\nu\to 0} \left( \nu \int_{0}^{t} \| \nabla u^{\nu} \|_{L^2}^2 \, d\tau -  \nu \int_{0}^{t} \| \nabla \theta^{\nu} \|_{L^2}^2 \, d\tau \right)=0, \qquad  t \in \mathcal{S},
\end{split}
\]
and hence \eqref{eq:Zero_Work} holds.

\end{proof}

We will also need to ensure that solutions of the NSE converge to a solution of the Euler equations in the above framework, which is guaranteed by the following standard argument.


\begin{lemma} \label{l:limiting_weak_solution}
Let $\{\theta^\nu\}_{\nu}$ be  bounded in $L^\infty(T_1,T_2; L^2)$ family of solutions to the advection-diffusion equation
\begin{equation} 
\p_t \theta^\nu + v^\nu \cdot \nabla \theta^\nu = \nu \Delta \theta^\nu,
\end{equation}
with drift $v^\nu \to v$ in $L^1(T_1,T_2;L^2)$ as $\nu \to 0$. Then there is a subsequence $\nu_j \to 0$ such that
\[
\theta^{\nu_j} \to \theta, \qquad \text{in} \qquad C_w([0,T];L^2),
\]
for some weak solution $\theta$ of the transport equation with drift $v$.
\end{lemma}
\begin{proof}
The existence of a convergent in $C_w([0,T];L^2)$ subsequence of $\{\theta^\nu\}_{\nu}$ follows from from the fact that the sequence of Fourier coefficients $\{\hat \theta^\nu_k(t)\}_k$ is uniformly in $\nu$  equicontinuous. Then by the Acoli-Arzela theorem, $\{\theta^\nu\}_{\nu}$ is relatively compact in $C_w([T_1,T_2];L^2)$. This also follows from the Aubin-Lions compactness theorem, which is a more traditional approach. Hence, there is a subsequence $\{\nu_m\}_m$ and  $\theta \in C_w([0,T];L^2)$ with
\[
\theta^{\nu_m} \to \theta, \qquad \text{in} \qquad C_w([0,T];L^2).
\]
To show that $\theta$ is a weak solution of the transport equation, we decompose
\[
v^{\nu_m} \theta^{\nu_m} - v \theta = (v^{\nu_m} -v)\theta^{\nu_m}+ v(\theta^{\nu_m}-\theta) 
\]
in the weak formulation and use $v^{\nu_m} \to v$ in $L^1(T_1,T_2;L^2)$ for the first term and the the weak convergence of $\theta^{\nu_m}$ to $\theta$ for the second.
\end{proof}

This lemma will imply that the constructed $2+\frac12$ dimensional solutions of \eqref{eq:NSE}  $u^{\nu}=(v^{\nu},\theta^\nu)$
with force $f^\nu \to f$ in $C([0,1];C^\alpha(\TT^2))$ will weakly converge, after passing to a subsequence, to a solution of the Euler equation with force $f$.

\section{Proof of Theorem~\ref{thm:main}}


\begin{proof}[Proof of Theorem~\ref{thm:main}]
For this construction we extend the smooth solution $u(t)$ of the Euler equation on $[0,1)$ defined in \eqref{eq:Def_u} to a weak solution on $[0,2]$ trivially. We keep the 2D velocity $v$ zero on $[1,2]$ and as a result the scalar $\rho(t)$ transported by $v$ stays constant for $t\geq 1$. Since the weak limit of $u(t)$ is zero as $t\to1-$ (see \eqref{eq:u_weak_limit}), we define
\[
v(t):=
\begin{cases}
\tilde v(t), & t\in[0,1),\\
0, & t \in[1,2],
\end{cases}
\qquad 
\rho(t):=
\begin{cases}
\tilde \rho(t), & t\in[0,1),\\
0, & t \in[1,2],
\end{cases}
\]
which are weakly continuous in $L^2$ (see \eqref{eq:def_tilde_v-limit}, \eqref{eq:def_tilde_rho-limit}). Then $\rho(t)$ is a weak solution (with one singular time $t=1$) of the transport equation
\begin{equation} 
\p_t \rho + v \cdot \nabla \rho =0,
\end{equation}
and
\[
u(t)=(v(t),\rho(t)),
\]
is a weak solution of the Euler equation with force $f=(g,0)=(\p_t v +(v \cdot \nabla) v,0)$ and initial data $u(0)=u_{\mathrm{in}}= (0,\rho_{\mathrm{in}})$ on the extended time interval $[0,2]$.

We will also need auxiliary functions
$v^m \in C^\infty( [0,2] \times \TT^2; \RR^2)$ and  $\rho^m \in C^\infty([0,2] \times \TT^2)$ defined as 
\[
v^m(t):=
\begin{cases}
\tilde v^m(t), & t\in[0,1),\\
0, & t \in[1,2],
\end{cases}
\qquad 
\rho^m(t):=
\begin{cases}
\tilde \rho^m(t), & t\in[0,1),\\
\rho^m(1), & t \in[1,2],
\end{cases}
\]
where $\tilde v^m$ and $\tilde \rho^m$ are defined in \eqref{eq:def_tilde_v} and \eqref{eq:def_tilde_rho} respectively.
Note that $\rho^m(t)$ satisfies the transport equation with drift $v^m(t)$ on
on the extended time interval $[0,2]$. Let $\theta^m$ be the unique smooth solution of the  advection diffusion equation
\begin{equation} 
\p_t \theta^m + v^m \cdot \nabla \theta^m = \nu_m \Delta \theta^m,
\end{equation}
on $[0,2]$ with $\theta^m(0)=\rho_{\mathrm{in}}$. The smooth solutions of the \eqref{eq:NSE} on $[0,2]$ will be given by
\[
u^{m}(t)=(v^{m}(t),\theta^m(t)),
\]
with force $f^m=(g^m,0)$, where $g^m$ is defined as in \eqref{eq:def_g^m}. Since $v^{m}(0)=0$ for every $m$, the initial data is given by $u^m(0)=u_{\mathrm{in}}= (0,\rho_{\mathrm{in}})$. 

\bigbreak

\noindent
{\bf Total dissipation anomaly on $[0,1]$.}
Since $\|\theta^m(0)\|_{L^2}=1$, by the energy equality for $\theta^m$,
\[
2\nu_m \int_{0}^{1} \|\nabla \theta^m\|_{L^2}^2 \, dt  \leq \|\theta^m(0)\|_{L^2}^2=1.
\]
Take any $T \in [t_{m+1},2]$, where recall $t_m=1-(m+1)^{-2}$.
Since the dissipation $\int_0^{T}\|\nabla \rho^m\|_{L^2}^2 \, dt$ is dominated by the last term in $\rho^{m}$ (i.e., $\rho_m$ in \eqref{eq:def_tilde_v}) supported on $[t_m, T]$, we have
\begin{equation} \label{eq:estimate_on_difference}
\begin{split}
\sup_{t\in[0, T]}\|\theta^m(t)-\rho^m(t)\|^2_{L^2} &\leq \left( 2\nu_m \int_{0}^{T} \|\nabla \rho^m\|_{L^2}^2 \, dt \right)^{\frac12} \left( 2\nu_m \int_{0}^{T} \|\nabla \theta^m\|_{L^2}^2 \, dt \right)^{\frac12} \\
&\leq \left( 2\nu_m \int_{0}^{T} \|\nabla \rho^m\|_{L^2}^2 \, dt \right)^{\frac12}\\
&\lesssim \sqrt{\nu_m (T-t_m)} \l_m.
\end{split}
\end{equation}
In particular, since $t_{m+1}-t_{m}\sim m^{-3}$, taking $T=t_{m+1}$ we obtain
\[
\sup_{t\in[0, t_{m+1}]}\|\theta^m(t)-\rho^m(t)\|^2_{L^2} \leq c \sqrt{\nu_m } m^{-\frac32} \l_m,
\] 
for some absolute constant $c$. For a sequence $0<\alpha_m<1$ with $\lim_{m\to \infty} \alpha_m =0$, fixed later, let the viscosity be
\begin{equation} \label{vis_h}
\nu_m= \nu_m^{\mathrm{h}} := \alpha_m^4 c^{-2} m^{3} \l_m^{-2} . 
\end{equation}
Then
\begin{equation} \label{eq:smallness_of_zeta}
\sup_{t\in[0, t_{m+1}]}\|\theta^m(t)-\rho^m(t)\|^2_{L^2} \leq \alpha_m^2.
\end{equation}
Denoting by $P_{\leq \Lambda_m}$ the projection on the frequency below 
\[
\Lambda_m := \frac{\alpha_m}{C} \l_m,
\]
 we note that, since the density $\rho^m$ stays on the highest frequency $\l_m$ on $[t_m,t_{m+1}]$, using estimates \eqref{eq:estimates_rho},
\begin{equation} \label{eq:TotalDA_rho^m_low_modes}
\begin{split}
\|P_{\leq \L_m}\rho^m(t)\|_{L^2} &\leq \Lambda_m \|P_{\leq \L_m}\rho^m(t)\|_{H^{-1}}\\
& \leq \frac{\alpha_m}{C} \l_m C \l_m^{-1}\\
& = \alpha_m,
\end{split}
\end{equation}
for $t\in [t_m,t_{m+1}]$. Thus, thanks to \eqref{eq:smallness_of_zeta}  we have 
\begin{equation} 
\begin{split}
\|P_{\leq \L_m}\theta^m(t_{m+1})\|_{L^2} &\leq   \|P_{\leq \L_m}(\theta^m(t_{m+1})-\rho^m(t_{m+1}))\|_{L^2}+ \|P_{\leq \L_m}\rho^m(t_{m+1})\|_{L^2}\\
&\leq \alpha_m+\alpha_m,
\end{split}
\end{equation}
Then using the fact that $v^m \equiv 0$ on $[t_{m+1}, 2]$, and hence $\theta^m(t)$ satisfies the heat equation,
\begin{equation} \label{eq:heat_thm1}
\p_t \theta^m = \nu_m \Delta \theta^m, \qquad t \in [t_{m+1}, 2],
\end{equation}
 we obtain, thanks to estimates \eqref{eq:estimates_rho},
\begin{equation} \label{eq:low_modes-Thm1}
\begin{split}
 \| P_{\leq \Lambda_m} \theta^m(t) \|_{L^2}^2 &\leq  \| P_{\leq \Lambda_m} \theta^m(t_{m+1}) \|_{L^2}^2\\
&\leq 4\alpha ^2_m\\
& \to 0,
\end{split}
\end{equation}
as $m\to \infty$ for all $t \in [t_{m+1},2]$.

Now turning to high modes, using again the fact that $\theta^m(t)$ satisfies the heat equation \eqref{eq:heat_thm1} and estimates \eqref{eq:estimates_rho}, we have
\[
\begin{split}
\frac{d}{dt} \|  P_{> \Lambda_m} \theta^m(t) \|_{L^2}^2 &= -2\nu_m \|\nabla  P_{> \Lambda_m} \theta^m(t)\|_{L^2}^2\\
& \leq  -2 \nu_m \Lambda_m^2\|P_{> \Lambda_m} \theta^m(t) \|_{L^2}^2.
\end{split}
\]
Hence, since $\|\theta^m(t_{m+1}) \|_{L^2}^2 \leq \|\theta^m(0) \|_{L^2}^2=1$, we obtain
\begin{equation} \label{eq:exponential_decay_on_high_modes}
\| P_{> \Lambda_m} \theta^m(t) \|_{L^2}^2 \leq e^{- 2\nu_m \Lambda_m^2 (t-t_{m+1})}, \qquad t\in [t_{m+1},2].
\end{equation}
Note that
\[
\begin{split}
\nu_m \Lambda_m^2 (1-t_{m+1}) &= \alpha_m^4 c^{-2} m^{3} \l_m^{-2} \cdot  \frac{\alpha_m^2}{C^2}\l_m^{2} \cdot (m+2)^{-2}\\
&\geq c_1 \alpha_m^6 m, 
\end{split}
\]
for some absolute constant $c_1$. So we choose $\alpha_m$ such that
\begin{equation} \label{eq:conditions_on_alpha_m}
\lim_{m\to \infty}\alpha_m = 0, \qquad \lim_{m\to \infty} \alpha_m^6 m = \infty,
\end{equation}
in which case \eqref{eq:exponential_decay_on_high_modes} implies 
\begin{equation} \label{eq:high_modes_limit}
\lim_{m\to \infty}  \| P_{> \Lambda_m} \theta^m(1) \|_{L^2}^2 =0.
\end{equation}
For instance, the choice of viscosity
\begin{equation} \label{vis_h-final}
\nu_m^{\mathrm{h}} = m^{\frac{5}{2}} \l_m^{-2}. 
\end{equation}
is compatible with \eqref{vis_h} and \eqref{eq:conditions_on_alpha_m}.

Now combining \eqref{eq:low_modes-Thm1} and \eqref{eq:high_modes_limit} we obtain
\begin{equation} \label{eq:vanishing}
\lim_{m \to \infty} \|\theta^m(1)\|_{L^2}^2 =0,
\end{equation}
and by the energy equality,
\begin{equation} \label{eq:Total_DA_Thm1}
 \lim_{m \to \infty} 2\nu_m \int_{0}^{1}\|\nabla \theta^m(t)\|^2_{L^2} \, dt =  \|\theta^m(0)\|_{L^2}^2 -   \lim_{m \to \infty} \|\theta^m(1)\|_{L^2}^2=1.
\end{equation}
So all the available energy dissipates in this case. The dissipation anomaly is the largest possible. 

\bigbreak

\noindent
{\bf No dissipation anomaly on $[0,1]$.} To determine the sequence of viscosities that does not lead to the dissipation anomaly, we recall that $1-t_{m}\sim m^{-2}$, and hence \eqref{eq:estimate_on_difference} yields
\begin{equation} \label{eq:difference_on_0-1}
\sup_{t\in[0, 1]}\|\theta^m(t) - \rho^m(t)\|^2_{L^2} \leq c_2 \sqrt{\nu_m } m^{-1} \l_m,
\end{equation}
for some absolute constant $c_2$. So for the choice of viscosity 
\begin{equation} \label{vis_int}
\nu_m= \nu_m^{\mathrm{int}}:= m\l_m^{-2},
\end{equation}
we have 
\begin{equation} \label{eq:No_DA_up_to_1_below_int}
\lim_{m \to \infty} \sup_{t\in[0, 1]}\|\theta^m(t) - \rho^m(t)\|^2_{L^2} =0,
\end{equation}
or, in particular,
\[
\lim_{m \to \infty}  \|\theta^m(1)\|_{L^2}^2 =1,
\]
and by the energy equality,
\begin{equation} \label{eq:No_DA_Thm1}
 \lim_{m \to \infty} 2\nu_m \int_{0}^{1}\|\nabla \theta^m(t)\|^2_{L^2} \, dt =  \|\theta^m(0)\|_{L^2}^2 -   \lim_{m \to \infty} \|\theta^m(1)\|_{L^2}^2 = 0.
\end{equation}
So there is no dissipation anomaly on $[0,1]$ in this case. 

\bigbreak

\noindent
{\bf Arbitrary dissipation anomaly on $[0,1]$.} Having obtained the total dissipation anomaly \eqref{eq:Total_DA_Thm1} for sequence $\nu_m = \nu_m^{\mathrm{h}}$ as well as no dissipation anomaly \eqref{eq:Total_DA_Thm1} for $\nu_m = \nu_m^{\mathrm{int}}$, we can find sequences of viscosities that result in all the possible levels of dissipation anomaly. Indeed,
using continuous dependence of $\|\theta^m(1)\|_{L^2}$ on the viscosity $\nu_m$, we conclude that 
for every $e \in [0,1]$, there is a sequence of solutions of the advection diffusion equation $\theta^{\nu^e_m}$ with $\nu^e_m \in[\nu_m^{\mathrm{int}}, \nu_m^{\mathrm{h}},]$ that 
dissipates $e$ amount of energy at time $t=1$ in the limit of vanishing viscosity:
\begin{equation} \label{Thm1_DA_1-proof}
\lim_{m\to \infty} 2\nu^e_m \int_{0}^{1} \| \nabla \theta^{\nu^e_m} \|_{L^2}^2 \, dt=e.
\end{equation}

\bigbreak

\noindent
{\bf Total dissipation anomaly on $[0,1+]$.} We fix any $t\in (1,2]$.
For any sequence of viscosities $\nu_m \in [\nu_m^{\mathrm{int}},\nu_m^{\mathrm{h}}]$, bound \eqref{eq:smallness_of_zeta} and hence \eqref{eq:low_modes-Thm1} still hold, i.e.,
\[
\lim_{m\to \infty} \| P_{\leq \Lambda_m} \theta^m(t) \|_{L^2}^2=0.
\]
On the other hand, bound \eqref{eq:exponential_decay_on_high_modes} gives
\[
\begin{split}
\| P_{> \Lambda_m} \theta^m(t) \|_{L^2}^2 &\leq e^{- 2\nu_m \Lambda_m^2 (t-t_{m+1})}\\
&\to 0,
\end{split}
\]
as $m \to \infty$ because $t-t_{m+1}$ is bounded away from zero and 
\[
\begin{split}
\nu_m \Lambda_m^2 \geq  \nu_m^{\mathrm{int}} \Lambda_m^2& = m \l_m^{-2} \cdot  \frac{\alpha_m^2}{C^2}\l_m^{2}\\
&=C^{-2} \alpha_m^6 m \to \infty,
\end{split}
\]
as $m\to \infty$ due to \eqref{eq:conditions_on_alpha_m}.

Hence we have 
\begin{equation} \label{Thm1_DA_2-proof}
\lim_{m \to \infty} \|\theta^m(t)\|_{L^2}^2 =0, \qquad \lim_{m\to \infty} 2\nu_m \int_{0}^{t} \| \nabla \theta^{\nu_m} \|_{L^2}^2 \, d\tau=1.
\end{equation}

\bigbreak

\noindent
{\bf Arbitrary dissipation anomaly on $[0,2]$ with continuous energy limit $E(t)$.} 
Now we examine a low range of viscosities $\nu_m \in [\nu_m^{\mathrm{l}},\nu_m^{\mathrm{int}}]$, where
\begin{equation} \label{vis_l}
\nu_m^{\mathrm{l}}:= m^{-1} \l_m^{-2}.
\end{equation}
Thanks to \eqref{eq:difference_on_0-1}, in this range of viscosities \eqref{eq:No_DA_up_to_1_below_int} holds,
i.e., there is no dissipation anomaly on $[0,1]$ and
\begin{equation} \label{eq:No-DA_up_to_1_cont_E}
\lim_{m \to \infty} \|\theta^m(1)-\rho^m(1)\|_{L^2}=0, \qquad \lim_{m \to \infty}  \|\theta^m(1)\|_{L^2}^2 =1,
\end{equation}

To show the existence of a subsequence with a continuous energy limit, first recall that $1-t_{m}\sim m^{-2}$, and hence by \eqref{eq:estimate_on_difference} we have
\begin{equation} \label{eq:estimates_differences_cont_E}
\|\theta^m(1) - \rho^m(1)\|_{L^2}^2 \leq c\sqrt{\nu_m} m^{-1}\l_m , \qquad \|\theta^m(1+\tau) - \rho^m(1+\tau)\|_{L^2}^2 \leq c\sqrt{\nu_m} (m^{-2} + \tau)^{\frac12}\l_m,
\end{equation}
for all $\tau \in [0,1]$. Since $\theta^m(t)$ satisfies the heat equation on $[1,2]$, the energy $\|\theta^m(t)\|_{L^2}^2$ is a convex decreasing uniformly continuous function on $[1,2]$. Therefore, using \eqref{eq:estimates_differences_cont_E} and the triangle inequality, for any  $t\in [1,2]$ and $\tau \in [0, 2-t]$ we obtain
\[
\begin{split}
\big| \|\theta^m(t)\|_{L^2} - \|\theta^m(t+\tau)\|_{L^2} \big| &\leq  \|\theta^m(1)\|_{L^2} - \|\theta^m(1+\tau)\|_{L^2}\\
 &= \|\theta^m(1)\|_{L^2} -\| \rho^m(1)\|_2 + \|\rho^m(1+\tau)\|_{L^2}- \|\theta^m(1+\tau)\|_{L^2}\\
&\leq \|\theta^m(1) - \rho^m(1)\|_2 +  \|\rho^m(1+\tau)-\theta^m(1+\tau)\|_{L^2}\\
& \leq 2 \sqrt{c} \nu_m^{\frac14} (m^{-2} + \tau)^{\frac14} \l_m^{\frac{1}{2}}\\
& \leq 2 \sqrt{c} k^{\frac14} (m^{-2} + \tau)^{\frac14},
\end{split}
\]
provided $\nu_m \leq k \l_m^{-2}$ for some $k>0$. Hence, for such viscosities, the sequence $\|\theta^m(t)\|_{L^2}$ is uniformly equicontinuous on $[1,2]$. Thus, by the Arzel\`a-Ascoli theorem, there is a subsequence $m_j \to \infty$ as $j\to \infty$ such that $\|\theta^{m_j}(t)\|_{L^2}$ converges uniformly, and hence
\begin{equation} \label{eq:E(t)_cont_proof}
E(t):= \lim_{j\to \infty} \|\theta^{m_j}(t)\|_{L^2},
\end{equation}
is continuous on $[1,2]$. On the other hand, we clearly have (as in \eqref{eq:estimate_on_difference} and  \eqref{eq:estimates_differences_cont_E}) that
\[
E(t) =\lim_{j\to \infty} \|\rho^{m_j}(t)\|_{L^2} =1, \qquad t \in[0,1].
\]
Hence 
\begin{equation} \label{eq:E(t)_continuous_Thm1_proof}
E \in C([0,2]).
\end{equation}
Finally, since $\|\theta^m(t)\|_{L^2}^2$ is a convex on $[1,2]$ for every $m$, we the limit energy $E(t)$ is convex on $[1,2]$, and, thanks to \eqref{eq:E(t)_continuous_Thm1_proof}, $E(t)$ is absolutely continuous on $[0,2]$.

To show that the dissipation anomaly on $[0,2]$ can be arbitrary (except total) in this regime, first note that when $\nu_m = \nu_m^{\mathrm{l}}$, bound \eqref{eq:estimate_on_difference} immediately implies that $\|\theta^m(2)\|_{L^2} \to \|\rho^m(2)\|_{L^2}=1$ as $m\to \infty$, and hence there is no dissipation anomaly on the whole interval $[0,2]$. The other extreme case is more delicate. The choice of viscosity $\nu_m^{\mathrm{int}}$ results in the full dissipation anomaly as we saw above. However, such a viscosity is too large to deduce the continuity of the energy limit $E(t)$. Therefore we consider the sequence
\[
\nu_m = k \l_m^{-2},
\]
for which we already know that the the energy limit \eqref{eq:E(t)_cont_proof} is continuous for some subsequence. Recall that there is no dissipation anomaly on $[0,1]$, i.e.,  \eqref{eq:No-DA_up_to_1_cont_E} holds. We choose frequency
\[
\Lambda_m := \frac{1}{k^{\frac14}C} \l_m,
\]
and as in \eqref{eq:TotalDA_rho^m_low_modes} notice that 
\[
\|P_{\leq \L_m}\rho^m(1)\|_{L^2} \leq k^{-\frac14}.
\]
Hence, using the facts that there is no dissipation anomaly on $[0,1]$, i.e.,  \eqref{eq:No-DA_up_to_1_cont_E} holds, as well as $v^m \equiv 0$ on $[1, 2]$, we obtain
\begin{equation} \label{eq:high_modes}
\begin{split}
 \| P_{\leq \Lambda_m} \theta^m(2) \|_{L^2} &\leq  \| P_{\leq \Lambda_m} \theta^m(1) \|_{L^2}\\
&\leq   \|P_{\leq \L_m}(\theta^m(1)-\rho^m(1))\|_{L^2}+ \|P_{\leq \L_m}\rho^m(1)\|_{L^2}\\
&\leq   \|\theta^m(1)-\rho^m(1)\|_{L^2}+ k^{-\frac14}\\
& \to k^{-\frac14},
\end{split}
\end{equation}
as $m\to \infty$, where we used \eqref{eq:No-DA_up_to_1_cont_E} to ensure the convergence of the bound.  Now recall from \eqref{eq:exponential_decay_on_high_modes} that the energy on high modes decays exponentially 
\[
\begin{split}
\| P_{> \Lambda_m} \theta^m(2) \|_{L^2}^2 &\leq e^{- 2\nu_m \Lambda_m^2 (2-1)}\\
&\leq e^{- 2\sqrt{k}}\\
&\to 0,
\end{split}
\]
as $m \to \infty$ by the choice of $\nu_m$ and $\Lambda_m$. Hence,
\[
\limsup_{m\to \infty} \| \theta^m(2) \|_{L^2}^2 \leq k^{-\frac14}, \qquad \text{and} \qquad \liminf_{m\to \infty} 2\nu_m \int_{0}^{2} \| \nabla \theta^{m} \|_{L^2}^2 \, d\tau\geq 1-k^{-\frac14} \quad \text{for} \quad \nu_m = k \l_m^{-2},
\]
due to the energy equality. This means the dissipation arbitrary can be arbitrary close to total (when $k$ is close to zero).

Recall that  by \eqref{eq:estimate_on_difference} there is no dissipation anomaly when $\nu_m = \nu_m^{\mathrm{l}}=m^{-1} \l_m^{-2}$, i.e.,
\[
 \liminf_{m\to \infty} 2\nu_m \int_{0}^{2} \| \nabla \theta^{m} \|_{L^2}^2 \, d\tau =0, \quad \text{for} \quad \nu_m = \nu_m^{\mathrm{l}}. 
\]
Then using continuous dependence of $\|\theta^m(2)\|_{L^2}$ on the viscosity $\nu_m$ we conclude that 
for every $e \in [0,1)$, there is a sequence of solutions of the advection diffusion equation $\theta^{\nu^e_m}$ with $\nu^e_m \in[\nu_m^{\mathrm{int}}, \nu_m^{\mathrm{h}},]$ that 
dissipates $e$ amount of energy at time $t=2$ in the limit of vanishing viscosity:
\begin{equation} \label{Thm1_DA_3-proof}
\lim_{m\to \infty} 2\nu^e_m \int_{0}^{2} \| \nabla \theta^{m} \|_{L^2}^2 \, dt=e;
\end{equation}

\bigbreak

\noindent
{\bf Conclusion. }
Finally, choose a sequence $m_j \to \infty$ fast enough, so that the ranges for the viscosity (see \eqref{vis_h-final} and \eqref{vis_l}) do not intersect:
\[
[\nu_{m_i}^{\mathrm{l}}, \nu_{m_i}^{\mathrm{h}}] \cap [\nu_{m_j}^{\mathrm{l}}, \nu_{m_j}^{\mathrm{h}}] = \emptyset, \qquad i\ne j.
\]
Then we define the family of viscosities $\mathcal{V}$ as the disjoint union of those intervals
\[
\mathcal{V}= \bigcup_{j \in \NN} [\nu_{m_i}^{\mathrm{l}}, \nu_{m_i}^{\mathrm{h}}].
\]
Now we can label the elements of the sequence $\{m_j\}$ by the elements of $\mathcal{V}$ by means of the mapping
 $\mathfrak{m}: \mathcal{V} \to \NN$ defined as
\[
\mathfrak{m}(\nu) = m_i \qquad \text{provided} \qquad \nu \in [\nu_{m_i}^{\mathrm{l}}, \nu_{m_i}^{\mathrm{h}}].
\]
Then we can label all the solutions of the advection diffusion equation considered above by the elements of $\mathcal{V}$. Now we switch to a more explicit notation and let $\theta^\nu$ be the solution of the advection diffusion equation with drift $v^{\mathfrak{m}(\nu)}$ and viscosity $\nu$.
Then the corresponding solutions to \eqref{eq:NSE} are defined as
\[
u^{\nu}:=(v^{\mathfrak{m}(\nu)},  \theta^\nu), \qquad \nu \in \mathcal{V}.
\]

Recall that $v^m(t) =0$ for all $m$ and $t\in[1,2]$, and hence
Lemma~\ref{l:DA_for_u=DA_for_th} implies
\[
\lim_{j\to \infty} \left( \nu_j \int_{0}^{t} \| \nabla u^{\nu_j} \|_{L^2}^2 \, d\tau - \nu_j \int_{0}^{t} \| \nabla \theta^{\nu_j} \|_{L^2}^2 \, d\tau\right)=0, \qquad \forall t\in[0,2],
\]
for any sequence $\nu_j \to 0$,  $\nu_j \in \mathcal{V}$ for all $j$.
Thus, the existence of subsequences of solutions with properties \eqref{Thm1_DA_1}, \eqref{Thm1_DA_2}, and  \eqref{Thm1_DA_3} follow from \eqref{Thm1_DA_1-proof}, \eqref{Thm1_DA_2-proof}, and \eqref{Thm1_DA_3-proof} respectively. The existence of the the second subfamily of solutions with absolutely continuous limiting energy $E(t)$ 
follows from \eqref{eq:E(t)_continuous_Thm1_proof} and \eqref{eq:u=theta_Therem1_proof}.

Finally,  the constructed family of the solutions $\{u^{\nu}\}_{\nu \in \mathcal{V}}$ to the Navier-Stokes equations can be reduced to a countable family
\[
\{u^{\nu}\}_{\nu \in \mathcal{V\cap \mathbb{Q}}},
\]
so that the statement of the Theorem~\ref{thm:main} still holds.

\end{proof}


\section{Proof of Theorem~\ref{thm:main-2}}

Now we proceed to proving the second main theorem involving a backward energy cascade.

\begin{proof}[Proof of Theorem~\ref{thm:main-2}]
We start with two explicit weak solutions of the Euler equations $u_1$ and $u_2$ on $[0,2]$, which are extensions of the smooth solution $u(t)$ on $[0,1)$ in Section~\ref{sec:Mixing}. 

\noindent
{\bf Two extreme limiting solutions of the Euler equations.}
For this construction we extend the drift $v(t)$ to $[1,2]$ by reversing time, which results in the reverse energy cascade for the scalar $\rho(t)$ for $t>1$. Immediately this gives two different explicit solutions to the transport and Euler equations. Indeed, let
\[
v(t):=
\begin{cases}
\tilde v(t), & t\in[0,1),\\
0, & t=1,\\
-\tilde v(2-t) , & t \in(1,2],
\end{cases}
\qquad 
\rho_1(t):=
\begin{cases}
\tilde \rho(t), & t\in[0,1),\\
0, & t=1,\\
\rho(2-t), & t \in(1,2],
\end{cases}
\qquad
\rho_2(t):=
\begin{cases}
\tilde \rho(t), & t\in[0,1),\\
0, & t \in[1,2],
\end{cases}
\]
which are weakly continuous in $L^2$ (see \eqref{eq:def_tilde_v-limit}, \eqref{eq:def_tilde_rho-limit}).
Then $\rho_1(t)$ and $\rho_2(t)$ are weak solutions (with singular time $t=1$) to the transport equation 
\begin{equation} 
\p_t \rho_i + v \cdot \nabla \rho_i =0, \qquad i=1,2,
\end{equation}
and
\begin{equation} \label{eq:two_limiting_solutions}
u_1(t)=(v(t),\rho_1(t)), \qquad u_2(t)=(v(t),\rho_2(t)),
\end{equation}
are weak solutions of the Euler equation with force $f=(g,0)=(\p_t v +(v \cdot \nabla) v,0)$ and initial data $u(0)=u_{\mathrm{in}}= (0,\rho_{\mathrm{in}})$ on the extended time interval $[0,2]$. 
Note that
\begin{equation} \label{eq:energy_of_rho1_and_rho2}
\|\rho_1(t)\|_{L^2}=
\begin{cases}
1, & t\in[0,1),\\
0, & t=1,\\
1, & t \in(1,2],
\end{cases}
\qquad
\rho_2(t):=
\begin{cases}
1, & t\in[0,1),\\
0, & t \in[1,2],
\end{cases}
\end{equation}
so clearly $\|u_1(t)\|_{L^2}> \|u_2(t)\|_{L^2}$ for all $t\in (1,2]$.

We will also need auxiliary functions
$v^m \in C^\infty(\TT^2\times [0,2]; \RR^2)$ and  $\rho^m \in C^\infty(\TT^2\times [0,2])$ be defined as 
\[
v^m(t):=
\begin{cases}
\tilde v^m(t), & t\in[0,1],\\
-\tilde v^m(2-t) , & t \in[1,2],
\end{cases}
\qquad 
\rho^m(t):=
\begin{cases}
\tilde \rho^m(t), & t\in[0,1],\\
\tilde \rho^m(2-t), & t \in[1,2],
\end{cases}
\]
where $\tilde v^m$ and $\tilde \rho^m$ are defined in \eqref{eq:def_tilde_v} and \eqref{eq:def_tilde_rho} respectively.
Then $\rho^m(t)$ satisfies the transport equation with drift $v^m(t)$ on
on the extended time interval $[0,2]$. Let $\theta^m$ be the unique smooth solution of the  advection diffusion equation
\begin{equation} 
\p_t \theta^m + v^m \cdot \nabla \theta^m = \nu_m \Delta \theta^m,
\end{equation}
on $[0,2]$ with $\theta^m(0)=\rho_{\mathrm{in}}$. The corresponding smooth solutions of the \eqref{eq:NSE} on $[0,2]$ are given by
\[
u^{m}(t)=(v^{m}(t),\theta^m(t)),
\]
with the force $f^m=(g^m,0)$, where $g^m$ is defined as in \eqref{eq:def_g^m}. Since $v^{m}(0)=0$ for every $m$, the initial data is given by $u^m(0)=u_{\mathrm{in}}= (0,\rho_{\mathrm{in}})$.

\bigbreak

\noindent
{\bf The first limiting solution: no dissipation anomaly.}
To obtain the first subsequence $\nu^1_m$, for which $u^{\nu^1_m}$ converges to the first limiting solution $u_1$ in \eqref{eq:two_limiting_solutions}, we  observe that
the dissipation $\int_0^2\|\nabla \rho^m\|_{L^2}^2 \, dt$ is dominated by the last term in $\rho^{m}$ (i.e., $\rho_m$ in \eqref{eq:def_tilde_v}) supported on the interval $[t_m, 2-t_m]$ of length $2(m+1)^{-2}$. Then, as in \eqref{eq:estimate_on_difference},
\begin{equation} \label{eq:improved_bound_difference}
\begin{split}
\sup_{t\in[0, 2]}\|\theta^m(t) - \rho^m(t)\|^2_{L^2} 
&\leq \left( 2\nu_m \int_{0}^{2} \|\nabla \rho^m\|_{L^2}^2 \, dt \right)^{\frac12}\\
&\leq c \sqrt{\nu_m } \l_m m^{-1},
\end{split}
\end{equation}
for some absolute constant $c$. So for the choice of viscosity 
\begin{equation} \label{vis_l-2}
\nu_m= \nu_m^{1}:= \l_m^{-2},
\end{equation}
we have 
\[
\lim_{m \to \infty}  \sup_{t\in[0, 2]}\|\theta^m(t) - \rho^m(t)\|^2_{L^2}=0,
\]
which implies that
\begin{equation} \label{eq:pfoof_thm_2_1st_seq}
\lim_{m \to \infty}  \|\theta^m(t)\|_{L^2} =\|\rho_{\mathrm{in}}\|_{L^2}=1, \qquad \forall t\in[0,2],
\end{equation}
and by the energy equality,
\[
 \lim_{m \to \infty} \nu_m \int_{0}^{2}\|\nabla \theta^m(t)\|^2_{L^2} \, dt =0.
\]
So there is no dissipation anomaly in this case on the whole extended interval $[0,2]$. Hence the choice $\nu_m = \nu^{1}_m$ \eqref{vis_l-2} gives us the first desired sequence of solutions to the NSE equations $u^{\nu^1_m}=(v^m, \theta^m)$ in Theorem~\ref{thm:main-2}. Indeed, thanks to Lemma~\ref{l:DA_for_u=DA_for_th}, 
\begin{equation} \label{eq:No_DA_second_theorem}
 \lim_{m \to \infty} \nu_m \int_{0}^{2}\|\nabla u^{\nu^1_m}(t)\|^2_{L^2} \, dt =0,
\end{equation}
so there is no dissipation anomaly. 

We remark that the choice of viscosity \eqref{vis_l-2} would result in a partial dissipation anomaly with continuous energy limit in the Theorem~\ref{thm:main} construction. The reverse energy cascade for $t>1$ resulting in the improved bound \eqref{eq:improved_bound_difference} is the reason that there is no dissipation anomaly for such large values of viscosity.

\bigbreak

\noindent
{\bf The second limiting solutions: total dissipation anomaly.}
Now we turn to the second sequence $\nu^2_m$. Thanks to Theorem~\ref{thm:main}, we know that there is a choice of viscosity, namely $\nu_m= \nu_m^{\mathrm{h}}=m^{\frac52} 2^{-2m}$ (see \eqref{vis_h-final} ), for which the dissipation anomaly is the largest possible on the time interval $[0,1]$:
\[
\lim_{m \to \infty} \|\theta^m(1)\|_{L^2} =0, \qquad  \lim_{m \to \infty} 2 \nu_m \int_{0}^{1}\|\nabla \theta^m(t)\|^2_{L^2} \, dt = {\textstyle \frac12} \|\theta^m(0)\|_{L^2}^2 =1.
\]
By the energy equality $\|\theta^m(t)\|_{L^2} \leq \|\theta^m(1)\|_{L^2} $ for every $m$ and $t \in [1,2]$, and hence 
\begin{equation} \label{control_theta_second_seq_Thm-2} 
\lim_{m \to \infty}  \|\theta^m(t)\|_{L^2} =0, \qquad \forall t \in[1,2].
\end{equation}
Then by the energy equality,
\begin{equation} 
 \lim_{m \to \infty} \nu_m \int_{0}^{2}\|\nabla \theta^m(t)\|^2_{L^2} \, dt =1.
\end{equation}
Hence choosing the viscosity $\nu^2_m = \nu^h_m$ we obtain the second sequence of the NSE solutions $u^{\nu^2_m}=(v^m, \theta^m)$ in Theorem~\ref{thm:main-2} (if necessary, we  pass to subsequences so that the sets of viscosity values do not intersect: $\{\nu^1_m\} \cap \{\nu^2_m\} = \emptyset$). 
By Lemma~\ref{l:DA_for_u=DA_for_th}, 
\begin{equation} \label{eq:Largest_DA_second theorem}
\lim_{m\to \infty} 2\nu_m \int_{0}^{2}\| u^{\nu^2_m}\|^2_{L^2} \, d\tau =\|u_{\mathrm{in}}\|_{L^2}^2=1,
\end{equation}
so this sequence dissipates all the available energy.

We can also see that $u_2$ is the limiting solution of the Euler equations for sequence $u^{\nu^2_m}$. Indeed, \eqref{control_theta_second_seq_Thm-2} implies that 
\[
u^{\nu^2_m} \to u_2, \quad \text{in} \quad C([1,2];L^2).
\]
As always, $u^{\nu^2_m} \to u_2$  in $C_w([0,1];L^2)$ and in $C([0,t];L^2)$ for all $0<t<1$.

\bigbreak

\noindent
{\bf Arbitrary dissipation anomaly on $[0,2]$.} Now using continuous dependence of $\|u^{\nu_m}(2)\|_{L^2}$ on $\nu_m$ as in the proof of the previous theorem, achieving two extreme levels of the dissipation anomaly, \eqref{eq:No_DA_second_theorem} (no dissipation anomaly) for $u^{\nu^1_m}$and \eqref{eq:Largest_DA_second theorem} (total dissipation anomaly) for $u^{\nu^2_m}$, 
for every $e \in [0,1]$, we can find a sequence of viscosities $\nu^e_m\to 0$ such that the corresponding sequence of solutions to the \eqref{eq:NSE} $u^{\nu^e_m}=(v^m,  \theta^m)$ dissipates $e$ amount of energy in the limit of vanishing viscosity:
\[
\lim_{m\to \infty} 2\nu^e_m \int_{0}^{2} \| \nabla u^{\nu^e_m} \|_{L^2}^2 \, dt=e.
\]
Therefore, as in the proof of the previous theorem, there is a decreasing sequence $\nu_m\to 0$ such that the set of the accumulation points of the sequence
$
 \nu_{m} \int_{0}^{2}\|\nabla u^{\nu_m}(t)\|^2_{L^2} \, dt
$  
is the interval $[0,1]$.

We remark that the argument above does not give us much information about the limiting solution for the sequence $u^{\nu_m}(t)$. By the weak convergence, the energy of the limiting solution can not be larger than $1-e$. However, at this point, the possibility that, for instance, for every $e\in (0,1]$, the limiting solution is $u_2(t)$, is not excluded. Our next goal is to prove that $u_1$ and $u_2$ are not the only limiting solutions.

\bigbreak


\noindent 
{\bf Infinitely many limiting solutions of the Euler equations.} Now we show that the backward energy cascade results in infinitely many limiting solutions of the Euler equations.

Thanks to bound \eqref{eq:improved_bound_difference}, for $0<\alpha<1$,
\begin{equation} \label{eq:smallness_of_zeta-infinitely-many-Euler-solutions}
\sup_{t\in[0, 2]}\|\theta^m(t)-\rho^m(t)\|^2_{L^2} \leq \alpha^2,
\end{equation}
provided
\begin{equation} \label{eq:nu_upper_bound-infinitely-many-Euler-solutions}
\nu_m \leq  \nu^{\alpha}_m:= \alpha^4c^{-2}\l_m^{-2} m^{2}.
\end{equation}

Denoting by $P_{\leq \L_m}$ the projection on the frequencies below 
\begin{equation} \label{eq:Inf_Euler_Lambda}
\Lambda_m := \frac{\sqrt{1-\alpha}}{C} \l_m,
\end{equation}
 we note that, since the density $\rho^m$ stays on the highest frequency $\l_m$ on $[t_m,2-t_{m}]$, by \eqref{eq:estimates_rho},
\[
\begin{split}
\|P_{\leq \L_m}\rho^m(t)\|_{L^2} &\leq \Lambda_m \|P_{\leq \L_m} \rho^m(t)\|_{H^{-1}}\\
& \leq \frac{\sqrt{1-\alpha}}{C} \l_m C \l_m^{-1}\\
& = \sqrt{1-\alpha},
\end{split}
\]
for $t\in [t_m,1]$.
Combined with $\|\rho^m(t)\|^2_{L^2} \equiv 1$ this means
\[
\|P_{> \L_m}\rho^m(t)\|_{L^2} \geq \sqrt{\alpha}, \qquad t\in [t_m,1].
\]
Thus, thanks to \eqref{eq:smallness_of_zeta-infinitely-many-Euler-solutions}, on $[t_m,1]$ we have, by the triangle inequality,  
\[
\begin{split}
\|P_{> \L_m}\theta^m(t)\|_{L^2} &\geq \|P_{> \L_m}\rho^m(t)\|_{L^2} - \|P_{> \L_m}(\rho^m(t)-\theta^m(t))\|_{L^2}\\
&\geq  \sqrt{\alpha}-\alpha,
\end{split}
\]
and hence, by the choice of $\Lambda_m$ in \eqref{eq:Inf_Euler_Lambda},
\begin{equation} \label{eq:lower_bound_on_grad_theta-infinitely-many-Euler-solutions}
\|\nabla \theta^m(t)\|^2_{L^2} \geq \frac{\l_m^{2}}{C^2}(1-\alpha)(\sqrt{\alpha}-\alpha)^2, \qquad t\in [t_m,1].
\end{equation}
Thus,
\[
\begin{split}
\nu_m \int_{0}^{1}\|\nabla \theta^m(t)\|^2_{L^2} \, dt &\geq
\nu_m \int_{t_m}^{1}\|\nabla \theta^m(t)\|^2_{L^2} \, dt\\
&\geq \frac{2\nu_m}{(m+1)^2} \frac{\l_m^{2}}{C^2}(1-\alpha)(\sqrt{\alpha}-\alpha)^2.
\end{split}
\]
So the choice of the viscosity $\nu_m=\nu^{\alpha}_m$ defined in \eqref{eq:nu_upper_bound-infinitely-many-Euler-solutions} yields
\[
2\nu_m \int_{0}^{1}\|\nabla \theta^m(t)\|^2_{L^2} \, dt \geq k\alpha^4 (1-\alpha)(\sqrt{\alpha}-\alpha)^2,
\]
for some absolute positive constant $k$.
Now thanks to the energy equality, for $t\in[1,2]$,
\begin{equation} \label{eq:bound_energy_inf_euler_proof}
\begin{split}
\|\theta^m(t)\|^2_{L^2} &\leq \|\theta^m(0)\|^2_{L^2} -  2\nu_m \int_{0}^{1}\|\nabla \theta^m(t)\|^2_{L^2} \, dt\\
 &\leq 1- k\alpha^4(1-\alpha)(\sqrt{\alpha}-\alpha)^2.
\end{split}
\end{equation}

Now by Lemma~\ref{l:limiting_weak_solution}, there is a subsequence $m_j \to \infty$, such that
\[
\theta^{m_j} \to \bar \theta \qquad  C_w([0,T];L^2),
\]
for some weak solution $\bar \theta$ of the transport equation
\[
\p_t \bar \theta + v \cdot \nabla \bar \theta =0.
\]
Due to the weak convergence of $\theta^{m_j}$ to $\bar \theta$, by \eqref{eq:bound_energy_inf_euler_proof} we have
\begin{equation} \label{eq:upper_bound_bar_theta-infinitely-many-Euler-solutions}
\begin{split}
\|\bar \theta(t)\|^2_{L^2} &\leq \liminf_{j \to \infty} \|\theta^{m_j}(t)\|^2_{L^2}\\
&\leq 1- k\alpha^4(1-\alpha)(\sqrt{\alpha}-\alpha)^2,
\end{split}
\end{equation}
for every $t\in[1,2]$.

On the other hand, recall that $\displaystyle\sup_{t\in[0,2]}\|\rho^{m_j}(t)-\theta^{m_j}(t)\|_{L^2} \leq \alpha$  by \eqref{eq:smallness_of_zeta-infinitely-many-Euler-solutions}, and hence, by the triangle inequality, 
\begin{equation} \label{eq:lower_bound_bar_theta-infinitely-many-Euler-solutions}
\begin{split}
\|P_{\leq \L}\bar\theta(t)\|_{L^2} & \geq
\|P_{\leq \L}\rho^m(t)\|_{L^2} - \|P_{\leq \L}(\rho^{m_j}(t)-\theta^{m_j}(t))\|_{L^2} - \|P_{\leq \L}(\theta^{m_j}(t)-\bar\theta(t))\|_{L^2}\\
&\geq \|P_{\leq \L}\rho^{m_j}(t)\|_{L^2} - \alpha - \|P_{\leq \L}(\theta^{m_j}(t)-\bar \theta(t))\|_{L^2},
\end{split}
\end{equation}
for all $t\in[0,2]$. By the convergence of $\rho^{m_j}$ to $\rho_1$ and $\theta^{m_j}$ to $\bar \theta$ in $C([0,2];L^2_w)$, 
\[
\lim_{j\to \infty}\|P_{\leq \L}\rho^{m_j}(t)\|_{L^2} = \|P_{\leq \L}\rho_1\|_{L^2}, \qquad \lim_{j\to \infty}\|P_{\leq \L}(\theta^{m_j}(t)-\bar \theta(t))\|_{L^2} =0.
\]
Then we first take the limit of \eqref{eq:lower_bound_bar_theta-infinitely-many-Euler-solutions} as $j\to \infty$ obtaining
\[
\|P_{\leq \L}\bar\theta(t)\|_{L^2}  \geq
\|P_{\leq \L}\rho_1(t)\|_{L^2} - \alpha, \qquad t\in[0,2],
\]
and then take the limit as $\L \to \infty$ to conclude that for all $t\in [0,1) \cup(1,2]$,
\begin{equation} \label{eq:lower_bound_bar_theta-infinitely-many-Euler-solutions-2}
\begin{split}
\|\bar\theta(t)\|_{L^2}  &\geq \|\rho_1(t)\|_{L^2} - \alpha\\
&\geq 1 - \alpha,
\end{split}
\end{equation}
where we used $\|\rho_1(t)\|_{L^2}\equiv1$ on $[0,1) \cup(1,2]$, see \eqref{eq:energy_of_rho1_and_rho2}.

Now combining bounds \eqref{eq:upper_bound_bar_theta-infinitely-many-Euler-solutions} and \eqref{eq:lower_bound_bar_theta-infinitely-many-Euler-solutions-2} we get
\begin{equation} \label{eq:double_bounds_on_bar_theta}
(1 - \alpha)^2 \leq \|\bar\theta(t)\|_{L^2}^2 \leq E(t) \leq 1- k\alpha^4(1-\alpha)(\sqrt{\alpha}-\alpha)^2, \qquad t \in(1,2],
\end{equation}
where
\[
E(t)=  \liminf_{m \to \infty} \|\theta^m(t)\|^2_{L^2}.
\]

We also know that the energy of the limiting solution $\|\bar \theta(t)\|^2_{L^2}$ is constant on $(1,2]$, but the above argument does not yield its exact value. It is plausible that $\theta^m$ does not converges strongly in $L^2$ to $\bar \theta$ on $(1,2]$ as some energy might get trapped on high modes (after reaching frequency $\l_m$) and does not come back to low modes with the rest of the solution. Nevertheless, \eqref{eq:double_bounds_on_bar_theta} implies the existence of infinitely many limiting solutions of the transport equation $\bar \theta$. For instance, choosing values of $\alpha$ along some sequence $\alpha_n \to 0$, $\alpha_n \in (0,1)$, $n=3,4,\dots$, thanks to \eqref{eq:double_bounds_on_bar_theta}, we get a sequence of solutions $\bar \theta_n(t)$ with $\|\bar \theta_n(2)\|^2_{L^2}<1$ for all $n$, and
\[
 \lim_{n\to \infty} \|\bar \theta_n(2)\|^2_{L^2} =1.
\]
The corresponding limiting weak solutions of the Euler equation are
\[
u_n(x,t)=
\begin{cases}
u_1(x,t), & t\in[0,1],\\
(v(x_1,x_2,t),\bar \theta_n), & t \in(1,2].
\end{cases}
\]
Since $v(2)=0$, the above solutions of the Euler equation also satisfy
\[
\|u_n(2)\|^2_{L^2}<1, \qquad  \lim_{n\to \infty} \|u_n(2)\|^2_{L^2} =1, \qquad n=3,4,\dots,
\] 
which concludes the proof of Theorem~\ref{thm:main-2}.
\end{proof}

\section{Dissipation Anomaly for long time averages: Proof of Theorem~\ref{thm:DA_in_FT}}


\begin{proof}[Proof of Theorem~\ref{thm:DA_in_FT}]
We start with the family $\tilde u^{\nu}$ of solution to \eqref{eq:NSE} from Theorem~\ref{thm:main} that exhibits the total dissipation anomaly on time interval $[0,1]$:
\begin{equation} \label{eq:DA_proof_DAT}
\lim_{\nu\to 0} 2\nu \int_{0}^{1} \| \nabla \tilde u^{\nu} \|_{L^2}^2 \, dt=1.
\end{equation}
In fact, all the dissipation anomaly occurs before the blow-up time, i.e., 
\begin{equation} \label{eq:DA_proof_DAT-t}
\lim_{\nu\to 0} 2\nu \int_{t}^{1} \| \nabla \tilde u^{\nu} \|_{L^2}^2 \, d\tau=1, \qquad \forall t \in[0,1).
\end{equation}
This family is constructed by taking $\tilde \theta^m$, the unique smooth solution of the  advection diffusion equation
\begin{equation} \label{eq:ADE_proof_DAT}
\p_t \tilde \theta^m + \tilde v^m \cdot \nabla \tilde \theta^m = \nu_m \Delta \tilde \theta^m,
\end{equation}
with $\tilde v^m$ defined in \eqref{eq:def_tilde_v}, initial data $\tilde \theta^m(0)=\rho_{\mathrm{in}}$, and an appropriate choice of viscosity $\nu_m = \nu_m^{\mathrm{h}}$ defined in  \eqref{vis_h-final}. Then
\begin{equation} \label{eq:DA_proof_DAT_theta}
\lim_{m\to \infty} \|\tilde \theta^m(1)\|_{L^2}=0, \qquad \lim_{m\to \infty} 2\nu_m \int_{t}^{1} \| \nabla \tilde \theta^m \|_{L^2}^2 \, d\tau=1, \qquad \forall t \in[0,1)
\end{equation}
and
\[
\tilde u^{\nu_m}(t) = (\tilde v^m(t), \tilde \theta^m(t)),
\]
gives the desired sequence of solutions to the \eqref{eq:NSE} satisfying \eqref{eq:DA_proof_DAT}. 

Now our goal is to extend $\tilde u^{\nu_m}(t)\chi_{[1-\frac{\tau}{4},1+\frac{\tau}{4}]}$ to the whole real line making sure that the resulting force converges in $L^\infty(\RR;L^2)$.

For $\tau \in(0,1)$ we choose a cutoff function $\eta \in C^\infty(\RR)$ such that
\begin{enumerate}
\item $\eta(t)=0$ for $t \in (-\infty, 1-\frac{\tau}{3}] \cup [1+\frac{\tau}{3}, \infty)$;
\item $\eta(t)=1$ for $t \in [1-\frac{\tau}{4},1+\frac{\tau}{4}]$,
\end{enumerate}
and a nondecreasing cutoff functions $\tilde \eta \in C^\infty(\RR)$ such that   
\begin{enumerate}
\item $\tilde \eta(t)=0$ for $t \in (-\infty, 1-\frac{\tau}{2}]$;
\item $\tilde\eta(t)=1$ for $t \in [1-\frac{\tau}{3},\infty]$.
\end{enumerate}

The density and the drift are defined as $\tau$-peroidizations of their cutoffs: 
\[
\theta^m(t):= \sum_{n \in \ZZ} \tilde \theta^m(t+\tau n)  \eta(t+\tau n), \qquad v^m(t):= \sum_{n \in \ZZ} \tilde v^m(t+\tau n) \tilde \eta(t+\tau n).
\]
Note that
\begin{equation} \label{eq:small_energy_drift}
\sup_m\|v^m\|_{L^\infty(\RR, L^2)} \to 0 \quad \text{as} \quad  \tau \to 0.
\end{equation}

Now the periodized density $\theta^m(t)$ still satisfies the advection diffusion equation, but with a force appearing due to the cutoff function $\eta(t)$:
\begin{equation} 
\p_t  \theta^m +  v^m \cdot  \nabla \theta^m = \nu_m \Delta  \theta^m + h^m.
\end{equation}
Note that
\[
\Supp h^m \chi_{[1-\frac{\tau}{2},1+\frac{\tau}{2}]} \subset \Supp \eta',
\]
for all $m$. Indeed, for $t\in[0,2]$,
\begin{equation} \label{eq:h^m_in_terms_of_eta'}
\begin{split}
h^m &= \p_t  \theta^m +  v^m \cdot  \nabla \theta^m - \nu_m \Delta  \theta^m\\
&=\eta'(t)\tilde\theta^m + \eta(t)(\p_t  \tilde\theta^m +  v^m \cdot  \nabla \tilde\theta^m - \nu_m \Delta  \tilde\theta^m)\\
&=\eta'(t)\tilde\theta^m.
\end{split}
\end{equation}
Note that $\Supp \eta' \subset [1-\frac{\tau}{3}, 1- \frac{\tau}{4}] \cup [1+\frac{\tau}{4}, 1+ \frac{\tau}{3}]$. We first focus on the time interval $[1-\frac{\tau}{3}, 1- \frac{\tau}{4}]$ where the density $\tilde \theta^m$ is under good control as it is still mostly on low modes. Recall that for the choice of viscosity used here $\nu_m = \nu_m^{\mathrm{h}}$ we have, thanks to \eqref{eq:smallness_of_zeta},
\[
\lim_{m\to \infty} \|\tilde \theta^m - \rho^m\|_{L^\infty(0,t;L^2)} =0,
\]
for every $t \in (0,1)$. Since for all large enough $m$,
\[
\rho^m \chi_{[1-\frac{\tau}{3}, 1- \frac{\tau}{4}]} = \tilde \rho \chi_{[1-\frac{\tau}{3}, 1- \frac{\tau}{4}]},
\]
for $\tilde \rho$ defined in \eqref{eq:def_tilde_rho-limit}, we have 
\[
h^m  \chi_{[1-\frac{\tau}{3}, 1- \frac{\tau}{4}]} \to \eta'(t) \tilde \rho \chi_{[1-\frac{\tau}{3}, 1- \frac{\tau}{4}]} \qquad \text{in} \quad L^{\infty}(1-{\textstyle \frac{\tau}{3}}, 1- {\textstyle \frac{\tau}{4}};L^2).
\]

Now let us consider the time interval $[1+\frac{\tau}{4}, 1+ \frac{\tau}{3}]$, where the density $\tilde \theta^m$ is already on high modes, but it is small thanks to the total dissipation anomaly. Using the expression for $h^m$ \eqref{eq:h^m_in_terms_of_eta'} we obtain
\[
\begin{split}
\|h ^m(t) \chi_{[1+\frac{\tau}{4}, 1+ \frac{\tau}{3}]}(t)\|_{L^2} &\leq \|\eta'\|_{L^\infty} \sup_{t\in [1+\frac{\tau}{4}, 1+ \frac{\tau}{3}]} \|\tilde\theta^m(t)\|_{L^2}\\
& \lesssim \|\tilde\theta^m(1)\|_{L^2}\\
& \to 0,
\end{split}
\]
as $m \to \infty$, where we used total dissipation anomaly \eqref{eq:DA_proof_DAT_theta}, which is the essential ingredient in this construction.

Now we can conclude that 
\[
h^m   \to h   \qquad \text{in} \quad L^{\infty}(\RR;L^2),
\]
where
\[
h= \sum_{n \in 2\ZZ} \eta'(t-n) \tilde \rho(t-n) \chi_{[1-\frac{\tau}{3}, 1- \frac{\tau}{4}]}(t-n).
\]

Finally, we define
\[
u^{\nu_m}(t) = (v^m(t), a_m\theta^m(t)),
\]
where the amplitudes $a_m>0$ are such that 
\[
\frac{1}{\tau}\int_{1-\frac{\tau}{2}}^{1+\frac{\tau}{2}} \|u^{\nu_m}\|_{L^2}^2\, dt = U^2, 
\]
for every $m$. This can be done since the energy of the drift can be made arbitrary small with a choice of small $\tau$ thanks to \eqref{eq:small_energy_drift}. It is easy to check that there exists $a>0$ (the amplitude for the limiting solution), such that
\[
\lim_{m\to \infty} a_m =a.
\]
Now we can verify that $u^{\nu_m}(t)$ is the desired sequence of time-periodic solutions to the \eqref{eq:NSE}. Indeed, since $\theta^{\nu_m}(t)$ coincides with $\tilde \theta^{\nu_m}(t)$ on $[1- \frac{\tau}{4}, 1]$,
we have, by \eqref{eq:DA_proof_DAT_theta},
\[
\begin{split}
\lim_{m\to \infty} 2\nu_m \lim_{T \to \infty} \frac{1}{T}\int_{0}^{T} \| \nabla  u^{\nu_m} \|_{L^2}^2 \, dt &= \lim_{m\to \infty} \frac{2\nu}{\tau} \int_{1-\frac{\tau}{2}}^{1+\frac{\tau}{2}} \| \nabla  u^{\nu_m} \|_{L^2}^2 \, dt\\
& = \lim_{m\to \infty} \frac{2\nu a_m^2}{\tau} \int_{1-\frac{\tau}{4}}^{1} \| \nabla  \theta^{\nu_m} \|_{L^2}^2 \, dt\\
&= \lim_{m\to \infty} \frac{2\nu a_m^2}{\tau} \int_{1-\frac{\tau}{4}}^{1} \| \nabla \tilde \theta^{\nu_m} \|_{L^2}^2 \, dt\\
&= \frac{a^2}{\tau}.
\end{split}
\]
 By choosing $\tau=a^2/(U^3 c)$ we get
\[
\frac{\epsilon \ell}{U^3} = c,
\]
with $\ell =1$. The size of the torus can be changed to arbitrary $\ell>0$ via standard rescaling.
\end{proof}

\section{Discontinuity of the limit: proof of Theorem~\ref{thm:discont-intro}}

The goal of this section is to prove that the dissipation anomaly implies the discontinuity of the limit solution
provided a certain frequency localization property is satisfied, reminiscent of Tao's delay mechanism in \cite{MR3486169}. 
Such a frequency localization property is enjoyed by quasi-selfsimilar constructions. More precisely,  we assume that the sequence of the solutions to \eqref{eq:NSE} $u^{\nu_j}(t)$ satisfies the following 
\begin{assumption} \label{assumption} There are sequences $u^{\nu_j}_{\mathrm{Loc}} \to u_{\mathrm{Loc}}$ in $C_w([0,1];L^2)$ and $u^{\nu_j}_{\mathrm{Ons}} \to u_{\mathrm{Ons}}$ in $L^\infty(0,1;L^2)$, such that
\[
u^{\nu_j}(t) = u^{\nu_j}_{\mathrm{Loc}}(t) + u^{\nu_j}_{\mathrm{Ons}}(t),
\]
where $u^{\nu_j}_{\mathrm{Loc}}(t)$ satisfies the following localization property:
\begin{equation} \label{eq:kernel}
\|\Delta_q u^{\nu_j}_{\mathrm{Loc}}(t)\|_{L^2} \leq c \lambda_{|q-\tilde q(j,t)|}^{-\alpha}, \qquad \forall j,q\in \NN, t\in[0,1],
\end{equation}
for some $c>0$, $\alpha>1$, and $\tilde q: \NN \times [0,1] \to \NN$. The remainder $u^{\nu_j}_{\mathrm{Ons}}$ does not exhibit dissipation anomaly
\begin{equation} \label{eq:no_anomaly_in_assumption}
\limsup_{j\to \infty} \nu_j \int_{0}^{1} \| \nabla u^{\nu_j}_{\mathrm{Ons}} \|_{L^2}^2 \, dt =0,
\end{equation}
and its limit belongs to the Onsager class $u_{\mathrm{Ons}} \in L^3(0,1;B^{\frac13}_{3,c_0}) $.
\end{assumption}
In \eqref{eq:kernel}, $\Delta_q u$ is the Littlewood-Paley projection of $u$, or just a rough projection on frequencies in the dyadic shell $[\l_{q-1},\l_{q+1}]$.

The main result is 

\begin{theorem} \label{thm:Discont_of_u}
Let $u^{\nu_j}(t)$ be a sequence of weak solutions to \eqref{eq:NSE} satisfying the energy equality with viscosity $\nu_j \to 0$ and force $f^{\nu_j} \to f$ in $L^1(0,1;L^2)$, converging weakly in $L^2$ to $u \in L^\infty(0,1;L^2)$
\[
u^{\nu_j} \to u \qquad \text{in} \qquad C_w([0,1];L^2),
\]
converging strongly at $t=0$
\[
u^{\nu_j}(0) \to u(0) \qquad \text{in} \qquad L^2,
\]
and exhibiting the dissipation anomaly, i.e.,
\begin{equation} \label{eq:diss_anomaly}
\limsup_{j\to \infty} \nu_j \int_0^1 \| \nabla u^{\nu_j} \|_{L^2}^2 \, dt >0.
\end{equation}
If in addition Assumption~\ref{assumption} holds, then $u(t)$ is discontinuous in $L^2$ at some $t\in[0,1]$.
\end{theorem}

\begin{remark}
In fact, the above result also holds for Leray-Hopf solutions, i.e., weak solutions satisfying the energy inequality starting from almost all initial data.
\end{remark}

The proof of the theorem relies on the following lemmas.

\begin{lemma} \label{lemma1}
Let $u^{\nu_j}(t)$ be a sequence of weak solutions to the NSE satisfying Assumption~\ref{assumption}  with viscosity $\nu_j \to 0$, converging weakly in $L^2$ to $u \in L^\infty(0,1;L^2)$
\[
u^{\nu_j} \to u \qquad \text{in} \qquad C_w([0,1];L^2).
\]
Then for every interval $[t_1,t_2] \in[0,1]$ exhibiting the dissipation anomaly, i.e.,
\begin{equation} \label{eq:diss_anomaly_lamma1}
\limsup_{j\to \infty} \nu_j \int_{t_1}^{t_2} \| \nabla u^{\nu_j} \|_{L^2}^2 \, dt >0,
\end{equation}
there exists $T\in [t_1,t_2]$ with
\[
u_{\mathrm{Loc}}(T)=0.
\]
\end{lemma}
\begin{proof}
Thanks to \eqref{eq:diss_anomaly_lamma1} and \eqref{eq:no_anomaly_in_assumption}, there exist $e>0$, a sequence $j_n \to \infty$, and $t_{j_n} \in [0,1]$, such that
\[
\nu_{j_n} \|\nabla u^{\nu_{j_n}}_{\mathrm{Loc}}(t_{j_n}) \|_{L^2}^2 \geq e, \qquad \forall n.
\]
By compactness of $[0,1]$, there exists $T \in [0,1]$, such that, passing to another subsequence and dropping subindexes, we have 
\begin{equation} \label{eq:large_diss}
t_j \to T \qquad \text{and} \qquad \nu_j \|\nabla u^{\nu_j}_{\mathrm{Loc}}(t_{j}) \|_{L^2}^2 \geq e, \qquad \forall j.
\end{equation}
Now using assumption \eqref{eq:kernel},
\[
\begin{split}
\|\nabla u^{\nu_j}_{\mathrm{Loc}}(t_{j}) \|_{L^2}^2 &\lesssim \sum_{q} \lambda_q^{2} \|\Delta_q u^{\nu_j}_{\mathrm{Loc}}(t_j)\|_{L^2}^2\\
&\lesssim \sum_{q}  \lambda_q^{2}c^2\lambda_{|q-\tilde q(j,t_j)|}^{-2\alpha} \\ 
&\lesssim \lambda_{\tilde q(j,t_j)}^2,  
\end{split}
\]
where we needed $\alpha >1$.
Combining this with \eqref{eq:large_diss} we obtain
\[
\lambda_{\tilde q(j,t_j)} \gtrsim \nu_j^{-\frac12}, \qquad \forall j.
\]
In particular, $\tilde q(j,t_j) \to \infty$ as $j\to \infty$. Now we will use the above estimate in \eqref{eq:kernel} to infer that for every fixed $q$ (and $j$ large enough such that   $q\leq \tilde q(j,t)$),
\begin{equation} \label{eq:convergence_to_0_AssumptionI}
\begin{split}
\|\Delta_q u^{\nu_j}_{\mathrm{Loc}}(t_j)\|_{L^2} &\lesssim  \lambda_{q}^\alpha \lambda_{\tilde q(j,t_j)}^{-\alpha}\\
&\lesssim \lambda_{q}^\alpha \nu_j^{\frac12 \alpha}\\
&\to 0,
\end{split}
\end{equation}
as $j \to \infty$. 

Finally, since $t_j \to T$ and $u^{\nu_j}_{\mathrm{Loc}} \to u_{\mathrm{Loc}}$ in $C_w([0,1];L^2)$, \eqref{eq:convergence_to_0_AssumptionI} implies that $\Delta_q u_{\mathrm{Loc}}(T) =0$ for every $q$, and hence $u_{\mathrm{Loc}}(T) = 0$.

\end{proof}

\begin{lemma} \label{lemma2}
Let $u^{\nu_j}(t)$ be a sequence of weak solutions to \eqref{eq:NSE} satisfying Assumption~\ref{assumption} converging weakly in $L^2$ to $u \in L^\infty(0,1;L^2)$
\[
u^{\nu_j} \to u \qquad \text{in} \qquad C_w([0,1];L^2).
\]
If
\begin{equation} \label{eq:energy_jump}
\limsup_{j\to \infty} \|u^{\nu_j}(T)\|_{L^2} > \|u(T)\|_{L^2},
\end{equation}
i.e., $u^{\nu_j}(T)$ does not converge to $u(T)$ strongly in $L^2$
at some $T\in[0,1]$, then
\[
u_{\mathrm{Loc}}(T)=0.
\]
\end{lemma}
\begin{proof}

If there exists $Q \in \NN$ such that $\tilde q(j,T) \leq Q$ for all $j$, then using \eqref{eq:kernel} to estimate the energy on high modes above $q^*>Q$ as
\[
\sum_{q\geq q^*} \|\Delta_q u_{\mathrm{Loc}}^{\nu_j}(T)\|_{L^2}^2 \leq \sum_{q\geq q^*} c^2  \lambda_{|q-\tilde q(j,T)|}^{-2\alpha} \lesssim \lambda_{q^*-Q}^{-2\alpha},
\]
for all $j$, we see that the weak convergence of $u^{\nu_j}_{\mathrm{Loc}}(T)$ to $u_{\mathrm{Loc}}(T)$ implies the strong convergence
\begin{equation} \label{eq:strong_conv_second_lemma_6_5}
u^{\nu_j}_{\mathrm{Loc}}(T) \to u_{\mathrm{Loc}}(T) \qquad \text{in} \quad L^2.
\end{equation}
Since by Assumption~\ref{assumption} $u^{\nu_j}_{\mathrm{Ons}}$ converges in $L^\infty(0,1;L^2)$, \eqref{eq:strong_conv_second_lemma_6_5} automatically implies the strong convergence $u^{\nu_j}(T) \to u(T)$ in $L^2$ contradicting \eqref{eq:energy_jump}. Hence
\[
\limsup_{j\to \infty} \tilde q(j,T) = \infty.
\]
This combined with \eqref{eq:kernel} yields
\[
\liminf_{j \to \infty} \|\Delta_{q_0}u^{\nu_j}_{\mathrm{Loc}}(T)\|_{L^2}  = 0,
\]
for any $ q_0 \in \NN$. Since $u^{\nu_j}_{\mathrm{Loc}}(T)$ converges weakly to $u_{\mathrm{Loc}}(T)$, we have 
\[
u_{\mathrm{Loc}}(T)=0.
\]
\end{proof}

We can now proceed with the proof of the main result.

\begin{proof}[Proof of Theorem~\ref{thm:Discont_of_u}]
First notice that the assumption $f^{\nu_j} \to f$ in $L^1(0,1;L^2)$ together with the weak convergence of $u^{\nu_j}$ to $u$ implies that
\begin{equation} \label{eq:Work_convergence}
\int_{t_0}^t (f^{\nu_j},u^{\nu_j}) \, d\tau = \int_{t_0}^t (f^{\nu_j}-f,u^{\nu_j}) \, d\tau + \int_{t_0}^t (f,u^{\nu_j}) \, d\tau  \to \int_{t_0}^t (f,u) \, d\tau, \qquad \forall 0 \leq t_0 \leq t \leq 1.
\end{equation}



Now suppose to the contrary that
\[
u \in C([0,1];L^2).
\]
Then we will prove the following.

\vspace{0.1in}
\noindent
{\bf Claim.} {\it There exists time $t_1\in[0,1]$ such that
\begin{equation} \label{eq:exists_t_1}
\limsup_{j\to \infty}\|u^{\nu_j}(t_1)\|_{L^2} >\|u(t_1)\|_{L^2}.
\end{equation}}
\vspace{0.1in}

This claim implies a contradiction. Indeed, we first note that Lemma~\ref{lemma2} implies $u_{\mathrm{Loc}}(t_1)=0$ and define
\[
t_0=\inf\{t\in [0,t_1]: u_{\mathrm{Loc}}(\tau)=0 \ \forall \ \tau \in[t,t_1]\}.
\]
Note that $t_0$ might be less than $t_1$, but by the energy $\|u^{\nu_j}(t)\|^2_{L^2}$ cannot have a large increase  on $[t_0,t_1]$ for large $j$. More precisely,  \eqref{eq:exists_t_1}  and the energy equality yield
\begin{equation} \label{eq:Energy_Equality_proof_main_discont}
\begin{split}
\|u(t_1)\|_{L^2}^2  & < \limsup_{j\to \infty}\|u^{\nu_j}(t_1)\|_{L^2}^2\\
 &\leq \limsup_{j\to \infty} \|u^{\nu_j}(t_0)\|_{L^2}^2 +   \lim_{j\to \infty}2\int_{t_0}^{t_1} (f^{\nu_j},u^{\nu_j}) \, d\tau\\
&= \limsup_{j\to \infty} \|u^{\nu_j}(t_0)\|_{L^2}^2  + 2\int_{t_0}^{t_1} (f,u) \, d\tau.
\end{split}
\end{equation}
On the other hand, on $[t_0,t_1]$, since $u_{\mathrm{Loc}}$ is zero, $u=u_{\mathrm{Ons}}$ satisfies the energy equality   as it belongs to the Onsager class $u \in L^3(t_0,t_1;B^{1/3}_{3,c_0})$ (see \cite{MR2422377}) , i.e.,
\[
\|u(t_1)\|_{L^2}^2 = \|u(t_0)\|_{L^2}^2 + 2\int_{t_0}^{t_1} (f,u) \, d\tau.
\]
Combining it with \eqref{eq:Energy_Equality_proof_main_discont} we obtain
\begin{equation} \label{eq:Negative_result_prove_main_inequality}
 \limsup_{j\to \infty} \|u^{\nu_j}(t_0)\|_{L^2}^2 > \|u(t_0)\|_{L^2}^2,
\end{equation}
so there is no strong convergence at $t_0$.

If $t_0=0$, then we get a contradiction with the strong convergence of the initial data. If $t_0>0$, then by definition of $t_0$ there exists a sequence $t_n \to t_0-$ as $n \to \infty$ with $u_{\mathrm{Loc}}(t_n) \ne 0$, and hence by Lemma~\ref{lemma2},
\[
\lim_{j\to \infty} \|u^{\nu_j}(t_n)\|_{L^2} = \|u(t_n)\|_{L^2},
\]
for every $n$. Employing the energy equality one more time,
\[
\begin{split}
\limsup_{j\to \infty}\|u^{\nu_j}(t_0)\|_{L^2}^2 &\leq \limsup_{j\to \infty} \|u^{\nu_j}(t_n)\|_{L^2}^2 +   \lim_{j\to \infty}2\int_{t_n}^{t_0} (f^{\nu_j},u^{\nu_j}) \, d\tau\\
&= \|u(t_n)\|_{L^2}^2  + 2\int_{t_n}^{t_0} (f,u) \, d\tau\\
&\to\|u(t_0)\|_{L^2}^2,
\end{split}
\]
as $n \to \infty$, where we used the continuity of $u(t)$ in $L^2$ to infer the convergence. Thus we obtained a contradiction with  \eqref{eq:Negative_result_prove_main_inequality}.

\vspace{0.1in}
\noindent
{\bf Proof of the claim.}
To prove the above claim, in view of Lemma~\ref{lemma1}, we consider three jointly exhaustive cases.

\vspace{0.1in}
\noindent
{\bf Case I.} 
There exist $0 \leq t_1 < t_2 \leq 1$ with
\[
\limsup_{j\to \infty} \nu_j\int_{t_1}^{t_2} \|\nabla u^{\nu_j}(\tau)\|_{L^2}^2 \, d\tau>0, \qquad u_{\mathrm{Loc}}(t_2)=0, \qquad \text{and} \qquad \|u_{\mathrm{Loc}}(t)\|_{L^2}>0 \qquad \forall t \in [t_1,t_2).
\]

\vspace{0.1in}
Then,  thanks to Lemma~\ref{lemma2},
\begin{equation} \label{eq:convergence_case1}
u^{\nu_j}(t) \to u(t) \qquad \forall t \in [t_1,t_2),
\end{equation}
in $L^2$. Also, due to Lemma~\ref{lemma1}, for any for any $t\in[t_1, t_2)$,
\[
\nu_j\int_{t_1}^{t} \|\nabla u^{\nu_j}(\tau)\|_{L^2}^2 \, d\tau \to 0, 
\]
as $j\to \infty$.
Then for all $t \in[t_1,t_2)$,
\[
\limsup_{j\to \infty} 2\nu_j\int_{t}^{t_2} \|\nabla u^{\nu_j}(\tau)\|_{L^2}^2 \, d\tau = \limsup_{j\to \infty} 2\nu_j\int_{t_1}^{t_2} \|\nabla u^{\nu_j}(\tau)\|_{L^2}^2 \, d\tau =:E >0.
\]

Now for any $t \in[t_1, t_2)$ the energy equality yields
\[
\begin{split}
\liminf_{j\to \infty}\|u^{\nu_j}(t_2)\|_{L^2}^2 &=  \lim_{j\to \infty}\|u^{\nu_j}(t)\|_{L^2}^2 - \limsup_{j \to \infty}2\nu_j\int_{t}^{t_2} \|\nabla u^{\nu_j}(\tau)\|_{L^2}^2 \, d\tau + \lim_{j\to \infty}2\int_{t}^{t_2} (f^{\nu_j},u^{\nu_j}) \, d\tau\\
&=\|u(t)\|_{L^2}^2 - E + 2\int_{t}^{t_2} (f,u) \, d\tau.
\end{split}
\]
Thanks to the weak convergence, $\displaystyle\|u(t_2)\|_{L^2}^2 \leq \liminf_{j\to \infty}\|u^{\nu_j}(t_2)\|_{L^2}^2$. Hence, taking the limit as $t \to t_2-$ and using the assumed continuity of $u(t)$ in $L^2$ we obtain
\[
\|u(t_2)\|_{L^2}^2 \leq \liminf_{j\to \infty}\|u^{\nu_j}(t_2)\|_{L^2}^2 = \|u(t_2)\|_{L^2}^2 - E < \|u(t_2)\|_{L^2}^2,
\]
a contradiction.

\vspace{0.1in}
\noindent
{\bf Case II.}
There exist $0 \leq t_1 < t_2 \leq 1$ with
\[
\limsup_{j\to \infty} \nu_j\int_{t_1}^{t_2} \|\nabla u^{\nu_j}(\tau)\|_{L^2}^2 \, d\tau>0, \qquad u_{\mathrm{Loc}}(t)(t_1)=0, \qquad \text{and} \qquad \|u_{\mathrm{Loc}}(t)(t)\|_{L^2}>0 \qquad \forall t \in (t_1,t_2].
\]

\vspace{0.1in}

Then,  thanks to Lemma~\ref{lemma2},
\begin{equation} \label{eq:convergence_case2}
u^{\nu_j}(t) \to u(t) \qquad \forall t_1 < t \leq t_2,
\end{equation}
in $L^2$. Also, due to Lemma~\ref{lemma1}, for any $t\in(t_1, t_2]$,
\[
\nu_j\int_{t}^{t_2} \|\nabla u^{\nu_j}(\tau)\|_{L^2}^2 \, d\tau \to 0, 
\]
as $j\to \infty$. Then for all $t \in(t_1,t_2]$,
\[
\limsup_{j\to \infty} 2\nu_j\int_{t_1}^{t} \|\nabla u^{\nu_j}(\tau)\|_{L^2}^2 \, d\tau = \limsup_{j\to \infty} 2\nu_j\int_{t_1}^{t_2} \|\nabla u^{\nu_j}(\tau)\|_{L^2}^2 \, d\tau =:E>0.
\]

Then for any $t \in(t_1, t_2]$ the energy equality yields
\[
\begin{split}
\|u(t)\|_2^2=\lim_{j\to \infty}\|u^{\nu_j}(t)\|_{L^2}^2 &=  \lim_{j\to \infty}\left(\|u^{\nu_j}(t_1)\|_{L^2}^2 - 2\nu_j\int_{t_1}^{t} \|\nabla u^{\nu_j}(\tau)\|_{L^2}^2 \, d\tau \right) + \lim_{j\to \infty}2\int_{t_1}^{t} (f^{\nu_j},u^{\nu_j}) \, d\tau\\
&=\limsup_{j\to \infty}\|u^{\nu_j}(t_1)\|_{L^2}^2 - E + 2\int_{t_1}^{t} (f,u) \, d\tau.
\end{split}
\]
Taking the limit as $t \to t_1+$ and using the assumed continuity of $u(t)$ in $L^2$ we obtain
\[
 \|u(t_1)\|_{L^2}^2  < \limsup_{j\to \infty} \|u^{\nu_j}(t_1)\|_{L^2}^2,
\]
and hence \eqref{eq:exists_t_1} holds.

\vspace{0.1in}

\noindent
{\bf Case III.} 
There exist $0 \leq t_1 < t_2 \leq 1$ with
\[
\limsup_{j\to \infty} \nu_j\int_{t_1}^{t_2} \|\nabla u^{\nu_j}(\tau)\|_{L^2}^2 \, d\tau>0 \qquad \text{and} \qquad u_{\mathrm{Loc}}(t)=0 \quad \forall t \in [t_1,t_2].
\]

\vspace{0.1in}

Then the energy equality on $[t_1,t_2]$ yields
\[
\begin{split}
\liminf_{j\to \infty}\|u^{\nu_j}(t_2)\|_{L^2}^2 &\leq  \limsup_{j\to \infty}\|u^{\nu_j}(t_1)\|_{L^2}^2 - \limsup_{j\to \infty}2\nu_j\int_{t_1}^{t_2} \|\nabla u^{\nu_j}(\tau)\|_{L^2}^2 \, d\tau  + \lim_{j\to \infty}2\int_{t_1}^{t_2} (f^{\nu_j},u^{\nu_j}) \, d\tau\\
&<\limsup_{j\to \infty}\|u^{\nu_j}(t_1)\|_{L^2}^2 + 2\int_{t_1}^{t_2} (f,u) \, d\tau.
\end{split}
\]
On the other hand  $u(t)$ satisfies the energy equality on $[t_1,t_2]$ as it belongs to the Onsager class $u \in L^\infty(t_1,t_2;B^{1/3}_{3,c_0})$, and hence we have
\[
\liminf_{j\to \infty}\|u^{\nu_j}(t_2)\|_{L^2}^2 \geq \|u(t_2)\|_{L^2}^2 = \|u(t_1)\|_{L^2}^2 + 2\int_{t_1}^{t_2} (f,u) \, d\tau.
\]

So again we have a time $t_1$ such that
\[
\|u(t_1)\|_{L^2}^2 <\limsup_{j\to \infty}\|u^{\nu_j}(t_1)\|_{L^2},
\]
and thus \eqref{eq:exists_t_1} holds.

\end{proof}

\section*{Acknowledgement }
The author is grateful to the hospitality of Princeton University during Spring 2023 where a portion of this work was completed, and Camillo De Lellis for stimulating discussions and pointing to the open problem of the total dissipation anomaly. 

\section*{Data Availability Statement} No datasets were generated or analysed during the current study.

\newpage


\section{Figures}

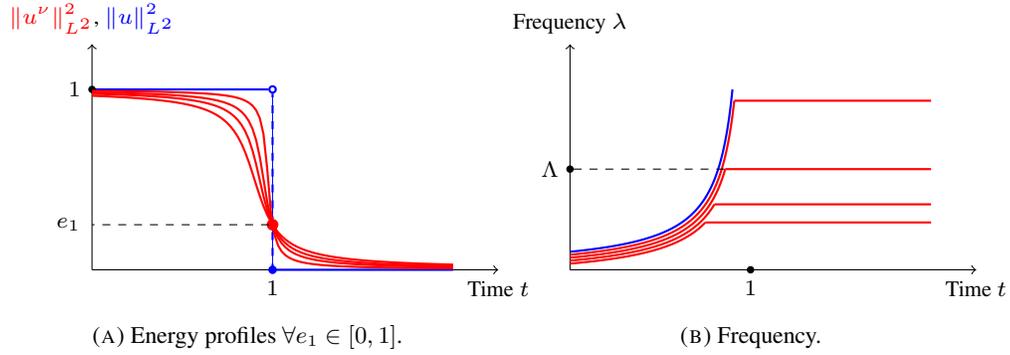
\begin{figure}[!h]
\centering 
\begin{subfigure}[b]{0.4\linewidth}
\begin{tikzpicture} [scale=1.2, every node/.style={transform shape}]

\draw[->] (0,0) -- (4.5,0) node[anchor=north] {\scriptsize Time $t$};
\draw[->] (0,0) -- (0,2.5) node[anchor=south ] {\scriptsize ${\color{red} \|u^\nu\|_{L^2}^2}$, ${\color{blue}\|u\|_{L^2}^2}$};

\draw[thick,color=blue]  (2,2) -- (0,2);
\filldraw  (0,2) circle (1pt)  node[left] {\scriptsize $1$};
\filldraw[thick,dashed, color=blue] (2,2) -- (2,0);
\filldraw (2,0) circle(0pt) node[anchor=north] {\scriptsize$1$};
\filldraw[thick,color=white]	(2,2) circle (1pt);
\draw[thick,color=blue] (2,2) circle (1pt);
\draw[thick, color=blue] (2,0) -- (4,0);
\filldraw[thick,color=blue] (2,0) circle (1pt);

\draw[thick,color=red]   plot[domain=0:4,  samples=100, id=f1] function{1+2*atan(5*(1.8-x))/3.1416};
\draw[thick,color=red]   plot[domain=0:4,  samples=100, id=f2] function{1+2*atan(7*(1.857-x))/3.1416};
\draw[thick,color=red]   plot[domain=0:4,  samples=100, id=f3] function{1+2*atan(10*(1.9-x))/3.1416} ;
\draw[thick,color=red]   plot[domain=0:4,  samples=100, id=f4] function{1+2*atan(20*(1.95-x))/3.1416};

\filldraw[thick,color=red]	(2,0.5) circle (1.5pt);
\draw[dashed]	 (1.96,0.5) -- (0,0.5) node[anchor=east] {\scriptsize $e_1$};

\end{tikzpicture}
\caption{Energy profiles  $\forall e_1 \in[0,1]$.}
\end{subfigure}
\begin{subfigure}[b]{0.4\linewidth}
\begin{tikzpicture} [scale=1.2, every node/.style={transform shape}]

\draw[->] (0,0) -- (4.5,0) node[anchor=north] {\scriptsize Time $t$};
\draw[->] (0,0) -- (0,2.5) node[anchor=south ] {\scriptsize Frequency $\lambda$};

\draw[thick,color=blue]   plot[domain=0:1.8, samples=100, id=f5] function{-0.4/(x-2)};
\filldraw (2,0) circle(1pt) node[anchor=north] {\scriptsize$1$};

\draw[thick,color=red]   plot[domain=0:1.82,  samples=100, id=f6] function{-0.03-0.4/(x-2.03)};

\draw[thick, color=red] (1.82,1.875) -- (4,1.875);

\draw[thick,color=red]   plot[domain=0:1.72,  samples=100, id=f7] function{-0.06-0.4/(x-2.06)};

\draw[dashed]   (1.71,1.1165) -- (0,1.1165);
\filldraw  (0,1.1165) circle (1pt)  node[left] {\scriptsize $\Lambda$};

\draw[thick, color=red] (1.72,1.1165) -- (4,1.1165);

\draw[thick,color=red]   plot[domain=0:1.6,  samples=100, id=f8] function{-0.09-0.4/(x-2.09)};

\draw[thick, color=red] (1.6,0.726) -- (4,0.726);

\draw[thick,color=red]   plot[domain=0:1.5,  samples=100, id=f9] function{-0.12-0.4/(x-2.12)};

\draw[thick, color=red] (1.5,0.525) -- (4,0.525);

\end{tikzpicture}
\caption{Frequency.}
\end{subfigure}
\caption{Theorem~\ref{thm:main}: Convergence of the solutions to the NSE (red) to a solution of the Euler equations (blue). The wavenumber $\Lambda\to \infty$ as $\nu \to 0$ schematically  shows the extend of the energy cascade for the NSE solutions.}
\label{fig:Thm1-convergence}
\end{figure} 


\begin{figure}[ht]
\centering 
\begin{subfigure}[b]{0.32\linewidth}
\begin{tikzpicture} [scale=1.1, every node/.style={transform shape}]

\draw[->] (0,0) -- (4.5,0) node[anchor=north] {\scriptsize Time $t$};
\draw[->] (0,0) -- (0,2.5) node[anchor=south ] {\scriptsize $E(t)$};

\draw[thick,color=red]  (2,2) -- (0,2);
\filldraw  (0,2) circle (1pt)  node[left] {\scriptsize $1$};
\filldraw[thick,dashed, color=red] (2,2) -- (2,0);
\filldraw (2,0) circle(0pt) node[anchor=north] {\scriptsize$1$};
\filldraw[thick,color=white]	(2,2) circle (1pt);
\draw[thick,color=red] (2,2) circle (1pt);
\draw[thick, color=red] (2,0) -- (4,0);
\filldraw[thick,color=red] (2,0.5) circle (1pt);
\filldraw[thick,color=white]	(2,0) circle (1pt);
\draw[thick,color=red] (2,0) circle (1pt);
\filldraw (4,0) circle(1pt) node[anchor=north] {\scriptsize$2$};

\draw[dashed]	 (1.96,0.5) -- (0,0.5) node[anchor=east] {\scriptsize $e_1$};

\end{tikzpicture}
\caption{$\Lambda \gg \kappa_d$.}
\end{subfigure}
\begin{subfigure}[b]{0.32\linewidth}
\begin{tikzpicture} [scale=1.1, every node/.style={transform shape}]

\draw[->] (0,0) -- (4.5,0) node[anchor=north] {\scriptsize Time $t$};
\draw[->] (0,0) -- (0,2.5) node[anchor=south ] {\scriptsize $E(t)$};

\draw[thick,color=red]  (2,2) -- (0,2);
\filldraw  (0,2) circle (1pt)  node[left] {\scriptsize $1$};
\filldraw (2,0) circle(1pt) node[anchor=north] {\scriptsize$1$};

\draw[thick,color=red]   plot[domain=2:4,  samples=100, id=g6] function{2*exp(-(x-2)/8)};

\draw[thick,color=red]   plot[domain=2:4,  samples=100, id=g4] function{2*exp(-(x-2)/4)};

\draw[thick,color=red]   plot[domain=2:4,  samples=100, id=g2] function{2*exp(-(x-2)/2)};

\draw[thick,color=red]   plot[domain=2:4,  samples=100, id=g1] function{2*exp(-(x-2))};

\draw[thick,color=red]   plot[domain=2:4,  samples=100, id=g3] function{2*exp(-(x-2)*2)};

\draw[thick,color=red]   plot[domain=2:4,  samples=100, id=g5] function{2*exp(-(x-2)*4)};

\draw[thick,color=red]   plot[domain=2:4,  samples=100, id=g7] function{2*exp(-(x-2)*8)};

\filldraw (4,0) circle(1pt) node[anchor=north] {\scriptsize$2$};

\draw[dashed]	 (4,0.7358) -- (0,0.7358) node[anchor=east] {\scriptsize $e_2$};
\end{tikzpicture}
\caption{$\Lambda \sim \kappa_d$, $e_2 \in (0,1]$.}
\end{subfigure}
\begin{subfigure}[b]{0.32\linewidth}
\begin{tikzpicture} [scale=1.1, every node/.style={transform shape}]

\draw[->] (0,0) -- (4.5,0) node[anchor=north] {\scriptsize Time $t$};
\draw[->] (0,0) -- (0,2.5) node[anchor=south ] {\scriptsize $E(t)$};

\draw[thick,color=red]  (4,2) -- (0,2);
\filldraw  (0,2) circle (1pt)  node[left] {\scriptsize $1$};
\filldraw (2,0) circle(1pt) node[anchor=north] {\scriptsize$1$};
\filldraw (4,0) circle(1pt) node[anchor=north] {\scriptsize$2$};


\end{tikzpicture}
\caption{$\Lambda \ll \kappa_d$.}
\end{subfigure}
\caption{Theorem~\ref{thm:main}: Energy profiles $E(t)$ for various subsequences of  solutions of the NSE. The wavenumber $\Lambda$ is heuristically compared to Kolmogorov's dissipation number $\kappa_d$.}
\label{fig:Thm1-energy_profiles}
\end{figure}
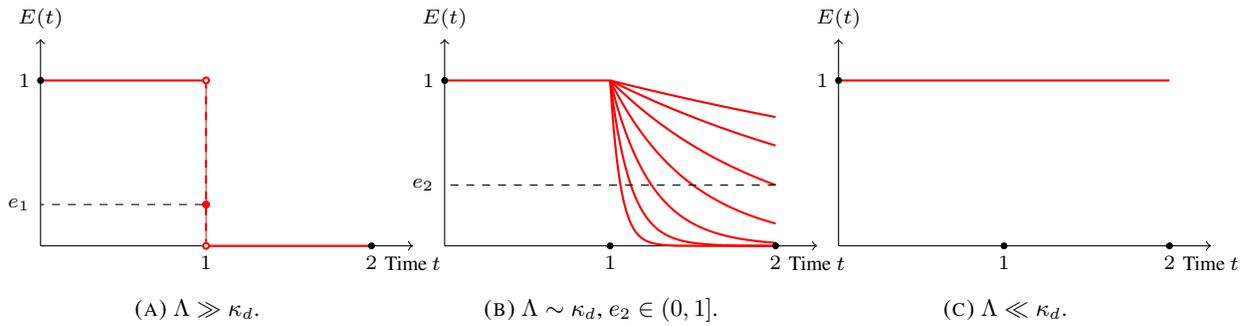




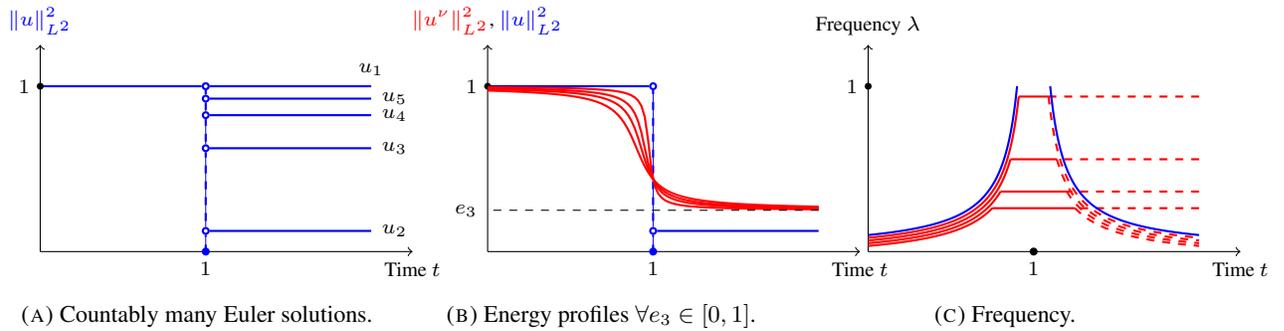
\begin{figure}[ht]
\centering 
\begin{subfigure}[b]{0.32\linewidth}
\begin{tikzpicture} [scale=1.1, every node/.style={transform shape}]

\draw[->] (0,0) -- (4.5,0) node[anchor=north] {\scriptsize Time $t$};
\draw[->] (0,0) -- (0,2.5) node[anchor=south ] {\scriptsize $\color{blue}\|u\|_{L^2}^2$};

\draw[thick,color=blue]  (2,2) -- (0,2);
\filldraw  (0,2) circle (1pt)  node[left] {\scriptsize $1$};
\filldraw[thick,dashed, color=blue] (2,2) -- (2,0);
\filldraw (2,0) circle(0pt) node[anchor=north] {\scriptsize$1$};
\filldraw[thick,color=white]	(2,2) circle (1pt);
\draw[thick,color=blue] (2,2) circle (1pt);
\filldraw[thick, color=blue] (2,0) circle (1pt);


\draw[thick, color=blue] (2,2) -- (4,2) node[anchor=south, color=black] {\scriptsize $u_1$};
\filldraw[thick,color=white]	(2,2) circle (1pt);
\draw[thick, color=blue]	(2,2) circle (1pt);

\draw[thick, color=blue] (2,1.85) -- (4,1.85) node[right, color=black] {\scriptsize $u_5$};
\filldraw[thick,color=white]	(2,1.85) circle (1pt);
\draw[thick, color=blue]	(2,1.85) circle (1pt);

\draw[thick, color=blue] (2,1.65) -- (4,1.65) node[right, color=black] {\scriptsize $u_4$};
\filldraw[thick,color=white]	(2,1.65) circle (1pt);
\draw[thick, color=blue]	(2,1.65) circle (1pt);

\draw[thick, color=blue] (2,1.25) -- (4,1.25) node[right, color=black] {\scriptsize $u_3$};
\filldraw[thick,color=white]	(2,1.25) circle (1pt);
\draw[thick, color=blue]	(2,1.25) circle (1pt);

\draw[thick, color=blue] (2,0.25) -- (4,0.25) node[right, color=black] {\scriptsize $u_2$};
\filldraw[thick,color=white]	(2,0.25) circle (1pt);
\draw[thick, color=blue]	(2,0.25) circle (1pt);



\end{tikzpicture}
\caption{Countably many Euler solutions.}
\label{fig:Thm2-countably_many}

\end{subfigure}
\begin{subfigure}[b]{0.32\linewidth}
\begin{tikzpicture} [scale=1.1, every node/.style={transform shape}]

\draw[->] (0,0) -- (4.5,0) node[anchor=north] {\scriptsize Time $t$};
\draw[->] (0,0) -- (0,2.5) node[anchor=south ] {\scriptsize $\color{red}\|u^\nu\|_{L^2}^2$, $\color{blue}\|u\|_{L^2}^2$};

\draw[thick,color=blue]  (2,2) -- (0,2);
\filldraw  (0,2) circle (1pt)  node[left] {\scriptsize $1$};
\filldraw[thick,dashed, color=blue] (2,2) -- (2,0);
\filldraw (2,0) circle(0pt) node[anchor=north] {\scriptsize$1$};
\filldraw[thick,color=white]	(2,2) circle (1pt);
\draw[thick,color=blue] (2,2) circle (1pt);
\draw[thick, color=blue] (2,0.25) -- (4,0.25);
\filldraw[thick, color=blue] (2,0) circle (1pt);

\draw[thick,color=red]   plot[domain=0:4,  samples=100, id=f10] function{1.25+1.5*atan(5*(1.8-x))/3.1416};
\draw[thick,color=red]   plot[domain=0:4,  samples=100, id=f11] function{1.25+1.5*atan(7*(1.857-x))/3.1416};
\draw[thick,color=red]   plot[domain=0:4,  samples=100, id=f12] function{1.25+1.5*atan(10*(1.9-x))/3.1416} ;
\draw[thick,color=red]   plot[domain=0:4,  samples=100, id=f13] function{1.25+1.5*atan(20*(1.95-x))/3.1416};

\filldraw[thick,color=white]	(2,0.25) circle (1pt);
\draw[thick, color=blue]	(2,0.25) circle (1pt);
\draw[dashed]	 (4,0.5) -- (0,0.5) node[anchor=east] {\scriptsize $e_3$};


\end{tikzpicture}
\caption{Energy profiles $\forall e_3 \in [0,1]$.}
\end{subfigure}
\begin{subfigure}[b]{0.32\linewidth}
\begin{tikzpicture} [scale=1.1, every node/.style={transform shape}]

\draw[->] (0,0) -- (4.5,0) node[anchor=north] {\scriptsize Time $t$};
\draw[->] (0,0) -- (0,2.5) node[anchor=south ] {\scriptsize Frequency $\lambda$};

\draw[thick,color=blue]   plot[domain=0:1.8, samples=100, id=f14] function{-0.4/(x-2)};
\filldraw  (0,2) circle (1pt)  node[left] {\scriptsize $1$};
\filldraw (2,0) circle(0pt) node[anchor=north] {\scriptsize$1$};
\filldraw[thick] (2,0) circle (1pt);


\draw[thick,color=red]   plot[domain=0:1.82,  samples=100, id=f16] function{-0.03-0.4/(x-2.03)};

\draw[thick,color=red,dashed]   plot[domain=2.18:4,  samples=100, id=f17] function{-0.03-0.4/(4-x-2.03)};

\draw[thick, color=red] (1.82,1.875) -- (2.18,1.875);
\draw[thick, color=red,dashed] (4,1.875) -- (2.18,1.875);

\draw[thick,color=red]   plot[domain=0:1.72,  samples=100, id=f18] function{-0.06-0.4/(x-2.06)};
\draw[thick,color=red,dashed]   plot[domain=2.28:4,  samples=100, id=f19] function{-0.06-0.4/(4-x-2.06)};

\draw[thick, color=red] (1.72,1.1165) -- (2.28,1.1165);
\draw[thick, color=red,dashed] (4,1.1165) -- (2.28,1.1165);

\draw[thick,color=red]   plot[domain=0:1.6,  samples=100, id=f20] function{-0.09-0.4/(x-2.09)};
\draw[thick,color=red,dashed]   plot[domain=2.4:4,  samples=100, id=f21] function{-0.09-0.4/(4-x-2.09)};

\draw[thick, color=red] (1.6,0.726) -- (2.4,0.726);
\draw[thick, color=red,dashed] (4,0.726) -- (2.4,0.726);

\draw[thick,color=red]   plot[domain=0:1.5,  samples=100, id=f22] function{-0.12-0.4/(x-2.12)};
\draw[thick,color=red,dashed]   plot[domain=2.5:4,  samples=100, id=f23] function{-0.12-0.4/(4-x-2.12)};

\draw[thick, color=red] (1.5,0.525) -- (2.5,0.525);
\draw[thick, color=red,dashed] (4,0.525) -- (2.5,0.525);

\draw[thick,color=blue]   plot[domain=2.2:4, samples=100, id=f15] function{0.4/(x-2)};

\end{tikzpicture}
\caption{Frequency.}
\end{subfigure}
\caption{Theorem~\ref{thm:main-2}: Convergence of the solutions of the NSE (red) to a solution of the Euler equation (blue).}
\label{fig:Thm2-convergence}
\end{figure} 




\pagebreak

\bibliographystyle{alpha}
\bibliography{Dissipation_Anomaly}

\newcommand{\etalchar}[1]{$^{#1}$}
\begin{thebibliography}{PKvdW02}

\bibitem[ACM19]{MR3904158}
Giovanni Alberti, Gianluca Crippa, and Anna~L. Mazzucato.
\newblock Exponential self-similar mixing by incompressible flows.
\newblock {\em J. Amer. Math. Soc.}, 32(2):445--490, 2019.

\bibitem[Aiz78]{Aizenman1978}
Michael Aizenman.
\newblock {\em On extensions of flows in the presence of sets of
  singularities}, pages 405--414.
\newblock Springer Berlin Heidelberg, Berlin, Heidelberg, 1978.

\bibitem[AV23]{armstrong2023anomalous}
Scott Armstrong and Vlad Vicol.
\newblock Anomalous diffusion by fractal homogenization, 2023.

\bibitem[BCC{\etalchar{+}}]{BrueEtc}
Elia Bru\`e, Maria Colombo, Gianluca Crippa, Camillo~De Lellis, and Massimo
  Sorella.
\newblock Onsager critical solutions of the forced navier-stokes equations.
\newblock {\em Communications on Pure and Applied Analysis}.

\bibitem[BCV22]{MR4422213}
Tristan Buckmaster, Maria Colombo, and Vlad Vicol.
\newblock Wild solutions of the {N}avier-{S}tokes equations whose singular sets
  in time have {H}ausdorff dimension strictly less than 1.
\newblock {\em J. Eur. Math. Soc. (JEMS)}, 24(9):3333--3378, 2022.

\bibitem[BDL23]{MR4595604}
Elia Bru\`e and Camillo De~Lellis.
\newblock Anomalous dissipation for the forced 3{D} {N}avier-{S}tokes
  equations.
\newblock {\em Comm. Math. Phys.}, 400(3):1507--1533, 2023.

\bibitem[BJW23]{2310.02934}
Jan Burczak, László~Székelyhidi Jr., and Bian Wu.
\newblock Anomalous dissipation and {E}uler flows.
\newblock {\em \eprint{2310.02934}}, 2023.

\bibitem[BLA05]{10.1063/1.2055529}
P.~Burattini, P.~Lavoie, and R.~A. Antonia.
\newblock {On the normalized turbulent energy dissipation rate}.
\newblock {\em Physics of Fluids}, 17(9):098103, 09 2005.

\bibitem[BMNV23]{MR4649134}
Tristan Buckmaster, Nader Masmoudi, Matthew Novack, and Vlad Vicol.
\newblock {\em Intermittent convex integration for the 3{D} {E}uler equations},
  volume 217 of {\em Annals of Mathematics Studies}.
\newblock Princeton University Press, Princeton, NJ, 2023.

\bibitem[BV19]{MR3898708}
Tristan Buckmaster and Vlad Vicol.
\newblock Nonuniqueness of weak solutions to the {N}avier-{S}tokes equation.
\newblock {\em Ann. of Math. (2)}, 189(1):101--144, 2019.

\bibitem[CCFS08]{MR2422377}
A.~Cheskidov, P.~Constantin, S.~Friedlander, and R.~Shvydkoy.
\newblock Energy conservation and {O}nsager's conjecture for the {E}uler
  equations.
\newblock {\em Nonlinearity}, 21(6):1233--1252, 2008.

\bibitem[CCS23]{MR4662772}
Maria Colombo, Gianluca Crippa, and Massimo Sorella.
\newblock Anomalous dissipation and lack of selection in the
  {O}bukhov-{C}orrsin theory of scalar turbulence.
\newblock {\em Ann. PDE}, 9(2):Paper No. 21, 48, 2023.

\bibitem[CDP07]{MR2337007}
Alexey Cheskidov, Charles~R. Doering, and Nikola~P. Petrov.
\newblock Energy dissipation in fractal-forced flow.
\newblock {\em J. Math. Phys.}, 48(6):065208, 10, 2007.

\bibitem[CET94]{MR1298949}
P.~Constantin, W.~E, and E.~S. Titi.
\newblock Onsager's conjecture on the energy conservation for solutions of
  {E}uler's equation.
\newblock {\em Comm. Math. Phys.}, 165(1):207--209, 1994.

\bibitem[CF09]{MR2522972}
Alexey Cheskidov and Susan Friedlander.
\newblock The vanishing viscosity limit for a dyadic model.
\newblock {\em Phys. D}, 238(8):783--787, 2009.

\bibitem[CFLS16]{MR3551263}
A.~Cheskidov, M.~C.~Lopes Filho, H.~J.~Nussenzveig Lopes, and R.~Shvydkoy.
\newblock Energy conservation in two-dimensional incompressible ideal fluids.
\newblock {\em Comm. Math. Phys.}, 348(1):129--143, 2016.

\bibitem[CFS10]{MR2665030}
Alexey Cheskidov, Susan Friedlander, and Roman Shvydkoy.
\newblock On the energy equality for weak solutions of the 3{D}
  {N}avier-{S}tokes equations.
\newblock In {\em Advances in mathematical fluid mechanics}, pages 171--175.
  Springer, Berlin, 2010.

\bibitem[CK15a]{MR3383921}
Alexey Cheskidov and Landon Kavlie.
\newblock Degenerate pullback attractors for the 3{D} {N}avier-{S}tokes
  equations.
\newblock {\em J. Math. Fluid Mech.}, 17(3):411--421, 2015.

\bibitem[CK15b]{MR3331677}
Alexey Cheskidov and Landon Kavlie.
\newblock Pullback attractors for generalized evolutionary systems.
\newblock {\em Discrete Contin. Dyn. Syst. Ser. B}, 20(3):749--779, 2015.

\bibitem[CL21]{MR4283701}
Alexey Cheskidov and Xiaoyutao Luo.
\newblock Anomalous dissipation, anomalous work, and energy balance for the
  {N}avier-{S}tokes equations.
\newblock {\em SIAM J. Math. Anal.}, 53(4):3856--3887, 2021.

\bibitem[CL22]{MR4462623}
Alexey Cheskidov and Xiaoyutao Luo.
\newblock Sharp nonuniqueness for the {N}avier-{S}tokes equations.
\newblock {\em Invent. Math.}, 229(3):987--1054, 2022.

\bibitem[CL23]{MR4610908}
Alexey Cheskidov and Xiaoyutao Luo.
\newblock {$L^2$}-critical nonuniqueness for the 2{D} {N}avier-{S}tokes
  equations.
\newblock {\em Ann. PDE}, 9(2):Paper No. 13, 56, 2023.

\bibitem[CS10]{MR2566571}
A.~Cheskidov and R.~Shvydkoy.
\newblock Ill-posedness of the basic equations of fluid dynamics in {B}esov
  spaces.
\newblock {\em Proc. Amer. Math. Soc.}, 138(3):1059--1067, 2010.

\bibitem[CS14]{MR3208714}
A.~Cheskidov and R.~Shvydkoy.
\newblock A unified approach to regularity problems for the 3{D}
  {N}avier-{S}tokes and {E}uler equations: the use of {K}olmogorov's
  dissipation range.
\newblock {\em J. Math. Fluid Mech.}, 16(2):263--273, 2014.

\bibitem[CS23]{MR4581460}
Alexey Cheskidov and Roman Shvydkoy.
\newblock Volumetric theory of intermittency in fully developed turbulence.
\newblock {\em Arch. Ration. Mech. Anal.}, 247(3):Paper No. 45, 35, 2023.

\bibitem[CTV14]{MR3187680}
Peter Constantin, Andrei Tarfulea, and Vlad Vicol.
\newblock Absence of anomalous dissipation of energy in forced two dimensional
  fluid equations.
\newblock {\em Arch. Ration. Mech. Anal.}, 212(3):875--903, 2014.

\bibitem[DC92]{PhysRevLett.69.1648}
Charles~R. Doering and Peter Constantin.
\newblock Energy dissipation in shear driven turbulence.
\newblock {\em Phys. Rev. Lett.}, 69:1648--1651, Sep 1992.

\bibitem[DEIJ22]{MR4381138}
Theodore~D. Drivas, Tarek~M. Elgindi, Gautam Iyer, and In-Jee Jeong.
\newblock Anomalous dissipation in passive scalar transport.
\newblock {\em Arch. Ration. Mech. Anal.}, 243(3):1151--1180, 2022.

\bibitem[Dep03]{MR2009116}
Nicolas Depauw.
\newblock Non unicit\'{e} des solutions born\'{e}es pour un champ de vecteurs
  {BV} en dehors d'un hyperplan.
\newblock {\em C. R. Math. Acad. Sci. Paris}, 337(4):249--252, 2003.

\bibitem[DF02]{MR1928940}
Charles~R. Doering and Ciprian Foias.
\newblock Energy dissipation in body-forced turbulence.
\newblock {\em J. Fluid Mech.}, 467:289--306, 2002.

\bibitem[DLS09]{MR2600877}
C.~De~Lellis and L.~Sz\'ekelyhidi, Jr.
\newblock The {E}uler equations as a differential inclusion.
\newblock {\em Ann. of Math. (2)}, 170(3):1417--1436, 2009.

\bibitem[DR00]{MR1734632}
Jean Duchon and Raoul Robert.
\newblock Inertial energy dissipation for weak solutions of incompressible
  {E}uler and {N}avier-{S}tokes equations.
\newblock {\em Nonlinearity}, 13(1):249--255, 2000.

\bibitem[Eyi94]{MR1302409}
Gregory~L. Eyink.
\newblock Energy dissipation without viscosity in ideal hydrodynamics. {I}.
  {F}ourier analysis and local energy transfer.
\newblock {\em Phys. D}, 78(3-4):222--240, 1994.

\bibitem[Foi97]{MR1467006}
Ciprian Foias.
\newblock What do the {N}avier-{S}tokes equations tell us about turbulence?
\newblock In {\em Harmonic analysis and nonlinear differential equations
  ({R}iverside, {CA}, 1995)}, volume 208 of {\em Contemp. Math.}, pages
  151--180. Amer. Math. Soc., Providence, RI, 1997.

\bibitem[Fri95]{Frisch}
U.~Frisch.
\newblock {\em Turbulence}.
\newblock Cambridge University Press, Cambridge, 1995.
\newblock The legacy of A. N. Kolmogorov.

\bibitem[GFGV12]{gomes-fernandes_ganapathisubramani_vassilicos_2012}
R.~Gomes-Fernandes, B.~Ganapathisubramani, and J.~C. Vassilicos.
\newblock Particle image velocimetry study of fractal-generated turbulence.
\newblock {\em Journal of Fluid Mechanics}, 711:306–336, 2012.

\bibitem[GKN23]{giri2023l3based}
Vikram Giri, Hyunju Kwon, and Matthew Novack.
\newblock The {$L^3$}-based strong {O}nsager theorem.
\newblock {\em \eprint{2305.18509}}, 2023.

\bibitem[GR23]{giri20232d}
Vikram Giri and Razvan-Octavian Radu.
\newblock The 2{D} {O}nsager conjecture: a {N}ewton-{N}ash iteration.
\newblock {\em \eprint{2305.18105}}, 2023.

\bibitem[Ise18]{MR3866888}
Philip Isett.
\newblock A proof of {O}nsager's conjecture.
\newblock {\em Ann. of Math. (2)}, 188(3):871--963, 2018.

\bibitem[JS23]{johansson2023nontrivial}
Carl Johan~Peter Johansson and Massimo Sorella.
\newblock Nontrivial absolutely continuous part of anomalous dissipation
  measures in time.
\newblock {\em \eprint{2303.09486}}, 2023.

\bibitem[JY21]{MR4297205}
In-Jee Jeong and Tsuyoshi Yoneda.
\newblock Vortex stretching and enhanced dissipation for the incompressible
  3{D} {N}avier-{S}tokes equations.
\newblock {\em Math. Ann.}, 380(3-4):2041--2072, 2021.

\bibitem[JY22]{MR4375721}
In-Jee Jeong and Tsuyoshi Yoneda.
\newblock Quasi-streamwise vortices and enhanced dissipation for incompressible
  3{D} {N}avier-{S}tokes equations.
\newblock {\em Proc. Amer. Math. Soc.}, 150(3):1279--1286, 2022.

\bibitem[KIY{\etalchar{+}}03]{10.1063/1.1539855}
Yukio Kaneda, Takashi Ishihara, Mitsuo Yokokawa, Ken'ichi Itakura, and Atsuya
  Uno.
\newblock {Energy dissipation rate and energy spectrum in high resolution
  direct numerical simulations of turbulence in a periodic box}.
\newblock {\em Physics of Fluids}, 15(2):L21--L24, 01 2003.

\bibitem[Kol41]{1371131421197965440}
A.~N. Kolmogorov.
\newblock The local structure of turbulence in incompressible viscous fluid for
  very large {R}eynolds numbers.
\newblock {\em Dokl. Akad. Nauk SSSR}, 30:301, 1941.

\bibitem[LMPP21]{MR4228012}
S.~Lanthaler, S.~Mishra, and C.~Par\'{e}s-Pulido.
\newblock On the conservation of energy in two-dimensional incompressible
  flows.
\newblock {\em Nonlinearity}, 34(2):1084--1135, 2021.

\bibitem[NV23]{MR4601999}
Matthew Novack and Vlad Vicol.
\newblock An intermittent {O}nsager theorem.
\newblock {\em Invent. Math.}, 233(1):223--323, 2023.

\bibitem[Ons49]{Onsager}
L.~Onsager.
\newblock Statistical hydrodynamics.
\newblock {\em Nuovo Cimento (9)}, 6(Supplemento, 2(Convegno Internazionale di
  Meccanica Statistica)):279--287, 1949.

\bibitem[PKvdW02]{10.1063/1.1445422}
B.~R. Pearson, P.-{\AA}. Krogstad, and W.~van~de Water.
\newblock {Measurements of the turbulent energy dissipation rate}.
\newblock {\em Physics of Fluids}, 14(3):1288--1290, 03 2002.

\bibitem[RDI23]{derosa2023support}
Luigi~De Rosa, Theodore~D. Drivas, and Marco Inversi.
\newblock On the support of anomalous dissipation measures.
\newblock {\em \eprint{2301.09603}}, 2023.

\bibitem[Sch93]{MR1231007}
Vladimir Scheffer.
\newblock An inviscid flow with compact support in space-time.
\newblock {\em J. Geom. Anal.}, 3(4):343--401, 1993.

\bibitem[Shn97]{MR1476315}
A.~Shnirelman.
\newblock On the nonuniqueness of weak solution of the {E}uler equation.
\newblock {\em Comm. Pure Appl. Math.}, 50(12):1261--1286, 1997.

\bibitem[Sre84]{10.1063/1.864731}
K.~R. Sreenivasan.
\newblock {On the scaling of the turbulence energy dissipation rate}.
\newblock {\em The Physics of Fluids}, 27(5):1048--1051, 05 1984.

\bibitem[Sre98]{10.1063/1.869575}
Katepalli~R. Sreenivasan.
\newblock {An update on the energy dissipation rate in isotropic turbulence}.
\newblock {\em Physics of Fluids}, 10(2):528--529, 02 1998.

\bibitem[Tao16]{MR3486169}
Terence Tao.
\newblock Finite time blowup for an averaged three-dimensional
  {N}avier-{S}tokes equation.
\newblock {\em J. Amer. Math. Soc.}, 29(3):601--674, 2016.

\bibitem[Vas15]{doi:10.1146/annurev-fluid-010814-014637}
J.~Christos Vassilicos.
\newblock Dissipation in turbulent flows.
\newblock {\em Annual Review of Fluid Mechanics}, 47(1):95--114, 2015.

\bibitem[Yud00]{MR1791984}
V.~I. Yudovich.
\newblock On the loss of smoothness of the solutions of the {E}uler equations
  and the inherent instability of flows of an ideal fluid.
\newblock {\em Chaos}, 10(3):705--719, 2000.

\end{thebibliography}

\end{document}